\documentclass[11pt,letterpaper]{amsart}
\usepackage[utf8]{inputenc}
\usepackage[T1]{fontenc}
\usepackage[all]{xy}
\usepackage{amsmath,amstext,amsthm,amssymb,amscd,version}
\usepackage{xcolor}
\usepackage{comment}
\usepackage{mathrsfs}

\def\bZ{\mathbb{Z}}
\def\bQ{\mathbb{Q}}
\def\bR{\mathbb{R}}
\def\bC{\mathbb{C}}
\def\bG{\mathbb{G}}

\def\bZ{\mathbb{Z}}

\def\cA{\mathcal{A}}

\def\cD{\mathcal{D}}
\def\cF{\mathcal{F}}

\def\cL{\mathcal{L}}
\def\cM{\mathcal{M}}

\def\cV{\mathcal{V}}

\theoremstyle{plain}

\newtheorem{thm}{Theorem}[section]
\newtheorem*{thm*}{Theorem}

\newtheorem{lem}[thm]{Lemma}
\newtheorem{prop}[thm]{Proposition}
\newtheorem{cor}[thm]{Corollary}

\newtheorem{defn}[thm]{Definition}

\newtheorem{thmintro}{Theorem}

\DeclareMathOperator{\Sp}{Sp}

\DeclareMathOperator{\rG}{G}
\DeclareMathOperator{\rH}{H}
\DeclareMathOperator{\rN}{N}

\DeclareMathOperator{\rS}{S}

\DeclareMathOperator{\GL}{GL}

\DeclareMathOperator{\Sh}{Sh}
\DeclareMathOperator{\sh}{sh}
\DeclareMathOperator{\Alb}{Alb}
\DeclareMathOperator{\alb}{alb}

\DeclareMathOperator{\HH}{H}

\DeclareMathOperator{\Sing}{Sing}
\DeclareMathOperator{\Reg}{Reg}

\newcommand{\sbt}{\,\begin{picture}(-1,1)(-1,-3)\circle*{3}\end{picture}\ }

\title{Hyperbolicity in presence of a large local system}
\author{Yohan Brunebarbe}
\address{Institut de Math\'ematiques de Bordeaux, Universit\'e de Bordeaux, 351 cours de la Lib\'eration, F-33405 Talence}
\email{yohan.brunebarbe@math.u-bordeaux.fr}

\begin{document}

\begin{abstract}
We prove that the projective complex algebraic varieties admitting a large complex local system satisfy a strong version of the Green-Griffiths-Lang conjecture.
\end{abstract}

\maketitle
\section{Introduction}

Let $X$ be a (non-necessarily smooth nor irreducible, but reduced) proper complex algebraic variety. Following Lang, we define the \emph{special} subsets $\Sp_{alg}(X)$, $\Sp_{ab}(X)$, and $\Sp_h(X)$ of $X$ as the union respectively of
\begin{itemize}
\item all (positive-dimensional) integral closed subvarieties not of general type;
\item the images of all non-constant rational maps $A \dasharrow X$ with source an abelian variety $A$;
\item all the entire curves of $X$, i.e. the image of all non-constant holomorphic maps $\bC \to X$.
\end{itemize}
It is not clear from their definition whether these subsets are Zariski-closed in $X$. The inclusions $\Sp_{ab}(X) \subset \Sp_{alg}(X)$ and $\Sp_{ab}(X) \subset \Sp_h(X)$ always hold, see Proposition \ref{trivial inclusions}. Inspired by an earlier conjecture of Green-Griffiths \cite{Green-Griffiths79}, Lang \cite{Lang86} conjectured that the Zariski closures of both $\Sp_{h}(X)$ and $\Sp_{ab}(X)$ are equal to $\Sp_{alg}(X)$, and that this locus is equal to $X$ if and only if $X$ is not of general type\footnote{One says that a non-necessarily irreducible projective variety is of general type if at least one of its irreducible component is.}. \\

In this paper, we establish that every projective complex algebraic variety that admits a large complex local system  satisfies the following strengthened form of the Green–Griffiths–Lang conjecture.

\begin{thmintro}\label{Lang for large}
Let $X$ be a projective complex algebraic variety. Assume that there exists a large complex local system on $X$. Then, 
\begin{enumerate}
\item The equality $\Sp_{alg}(X) = \Sp_{ab}(X) = \Sp_h(X)$ holds; therefore, the special subset is denoted $\Sp(X)$ without risk of confusion.
\item $ \Sp(X)$ is Zariski-closed in $X$;
\item $ \Sp(X) \neq X$ if and only if $X$ is of general type.
\end{enumerate} 
\end{thmintro}

Recall that a complex local system $\cL$ on a proper complex algebraic variety is large if for every integral closed subvariety $Z \hookrightarrow X$ the pull-back of $\cL$ to the normalization of $Z$ is non-isotrivial\footnote{A local system on an algebraic variety $X$ is called isotrivial if it becomes trivial on a finite étale cover of $X$.}. Equivalently, the Galois \'etale cover $\tilde{X}^\cL \to X$ associated to the kernel of the monodromy representation of $\cL$ does not have any positive-dimensional compact complex subspaces, cf. Proposition \ref{covering associated to a large local system}. This holds, for example, when the complex analytic space $ \tilde{X}^\cL$ is Stein, and turns out to be equivalent at least when $X$ is smooth \cite{Eyssidieux-reductive, EKPR}. \\

There is also the slightly weaker notion of bigness for complex local systems, which has the advantage of being a birationally invariant property. A complex local system $\cL$ on a connected normal projective complex algebraic variety is said to be big (or generically large in the terminology of Koll\'ar \cite{Kollar-Shafarevich}) if there exists a countable collection of proper closed subvarieties $D_i \subsetneq X$ such that, for every irreducible closed subvariety $Z \hookrightarrow X$ not contained in $\cup_i D_i$, the pull-back of $\cL$ to the normalization of $Z$ is non-isotrivial.

\begin{thmintro}\label{Lang for big}
Let $X$ be a connected normal projective complex algebraic variety. Assume that there exists a big complex local system on $X$. Then, the following assertions are equivalent:
\begin{enumerate}
\item $X$ is of general type;
\item $\Sp_{alg}(X)$ is not Zariski-dense in $X$; 
\item $\Sp_{ab}(X)$ is not Zariski-dense in $X$;
\item $\Sp_{h}(X)$ is not Zariski-dense in $X$.
\end{enumerate}     
\end{thmintro}

Thanks to \cite[Theorem 1]{Zuo-Chern-hyperbolicity} and \cite[Theorem 1]{CCE}, a connected normal projective complex algebraic variety $X$ equipped with a big complex local system whose algebraic monodromy group is semisimple is of general type. Therefore, Theorem \ref{Lang for big} ensures that in this setting the special sets are not Zariski-dense. \\

Examples of projective complex algebraic varieties admitting a large complex local system include:
\begin{enumerate}
    \item Projective complex algebraic varieties admitting a finite morphism to an abelian variety. In that case, Theorem \ref{Lang for large} follows from the works of Bloch \cite{Bloch26}, Ueno \cite{Ueno}, Ochiai \cite{Ochiai77}, Kawamata \cite{Kawamata80} and Yamanoi \cite{Yamanoi-maximal-Albanese}, see \cite{Bruni-relative-Lang}.
    \item Projective complex algebraic varieties admitting a (graded-polarizable) variation of $\bZ$-mixed Hodge structure with a finite period map. When in addition the Hodge structures are pure, it follows from the works of Griffiths and Schmid \cite{Griffiths-Schmid} that $ \Sp_{alg}(X) = \Sp_{ab}(X) = \Sp_h(X) = \varnothing$. In the mixed case, however, one can get nonempty special subsets, and this makes the proof of Theorem \ref{Lang for large} significantly more difficult. 
    \item The previous class includes all projective complex algebraic varieties admitting a finite morphism to the universal principally polarized abelian variety of dimension $g \geq 1$. The validity of Theorem \ref{Lang for large} in this case is new to our knowledge.
\end{enumerate}

In essence, the proof of Theorem \ref{Lang for large} and Theorem \ref{Lang for big} reduces the general case to the special cases above, relying on general structural results from non-abelian Hodge theory. This approach is reminiscent of the strategy introduced by Zuo \cite{Zuo-Chern-hyperbolicity}, where he shows that any smooth projective complex variety equipped with a big complex local system whose algebraic monodromy group is almost-simple admits a proper closed subvariety containing all curves of geometric genus $\leq 1$; see also \cite{Yamanoi-Fourier, CCE, Javanpeykar-Rousseau} for related hyperbolicity results in the presence of a local system. However, proving Theorem \ref{Lang for large} and Theorem \ref{Lang for big} requires substantially more precise results, as we must both handle subvarieties of dimension $>1$ and avoid assuming that the algebraic monodromy group is semisimple. In the non-semisimple case, a key ingredient is a result on the behaviour of special subsets in families of varieties of maximal Albanese dimension proved in \cite{Bruni-relative-Lang}.\\

About six months after the appearance of our work on the arXiv, Cadorel--Deng--Yamanoi \cite{CDY}, relying on recent progress on the existence of harmonic maps into Bruhat–Tits buildings, extended Theorem~\ref{Lang for big} to the setting of smooth quasi-projective complex algebraic varieties, under the additional assumption that the algebraic monodromy group is semisimple. The complete generalization of Theorem~\ref{Lang for big}, and a fortiori of Theorem~\ref{Lang for large}, remains open.

 \subsubsection*{Acknowledgements} We thank Marco Maculan for some useful discussions and comments. Thanks also to Stéphane Druel for his interest.

\subsubsection*{Conventions} A \emph{complex algebraic variety} is a separated reduced finite type $\bC$-scheme. One often makes no distinction between a complex algebraic variety and the associated complex analytic space.\\
A \emph{fibration} between two normal complex algebraic varieties is a proper surjective morphism $X \to Y$ with connected fibers.

\section{Generalities on special sets}

We gather some easy properties of special sets for future reference.

\begin{prop}\label{trivial inclusions}
Let $X$ be a proper complex algebraic variety. Then, 
\[\Sp_{ab}(X) \subset \Sp_{alg}(X) \text{ and }\Sp_{ab}(X) \subset \Sp_h(X).\]
\end{prop}

\begin{proof}
The inclusion $\Sp_{ab}(X) \subset \Sp_{alg}(X)$ follows from the fact that the image of a non-constant rational map $A \dasharrow X$ from an abelian variety $A$ is not of general type (this is for example an easy consequence of a special case of Iitaka conjecture proved by Viehweg, cf. \cite[Corollary IV]{Viehweg83}; this can also be checked more directly as a consequence of the triviality of the cotangent bundle of $A$).
On the other hand, for every rational map $Y \dasharrow X$ with $Y$ smooth projective, there exists a sequence of blow-ups $Y^\prime \to Y$ along smooth subvarieties such that the composite rational map $Y^\prime \to Y \dasharrow X$ is defined everywhere. Note that the exceptional locus of $Y^\prime \to Y$ is covered by rational curves. The inclusion $\Sp_{ab}(X) \subset \Sp_h(X)$ follows, since every abelian variety is covered by entire curves.
\end{proof}

\begin{prop}\label{Special sets and finite étale cover}
Let $f \colon X \to Y$ be a finite étale cover between proper complex algebraic varieties. Then, for any $\ast \in \{ alg, ab, h \}$,
\[ \Sp_\ast(X) = f^{-1}\left( \Sp_\ast(Y) \right) . \]
\end{prop}

\begin{prop}\label{Special sets and birational morphism}
Let $f \colon X \to Y$ be a birational morphism between irreducible proper complex algebraic varieties. Then, for any $\ast \in \{ alg, ab, h \}$, $\Sp_\ast(X)$ is Zariski-dense in $X$ if and only if $\Sp_\ast(Y)$ is Zariski-dense in $Y$.
\end{prop}
\begin{proof}
Let $Z \subsetneq Y$ be a closed subvariety such that $f$ is an isomorphism over $Y \backslash Z$. Then $ f \left( \Sp_\ast(X) \right) \backslash Z = \Sp_\ast(Y)  \backslash Z $, and the result follows.
\end{proof}

\begin{defn} 
Let $X \to Y$ be a proper morphism between complex algebraic varieties. For any $\ast \in \{ alg, ab, h \}$, we let
\[ \Sp_\ast(X/Y) := \bigcup_{y \in Y} \Sp_\ast(X_y) . \]
\end{defn}

\begin{prop}\label{Special sets and fibrations}
Let $f \colon X \to Y$ be a morphism between proper complex algebraic varieties. Then, for any $\ast \in \{ alg, ab, h \}$,
\[ \Sp_\ast(X) \subset f^{-1} \left( \Sp_\ast(Y) \right)  \cup \Sp_\ast(X/Y) . \]
\end{prop}

\begin{prop}\label{product-special-sets}
Let $n$ be a positive integer and, for $1 \leq i \leq n$, let $S_i$ be a projective complex algebraic variety. For every $i$, let $p_i \colon \prod_{1 \leq k \leq n} S_k \rightarrow S_i$ be the projection on the $i$-th factor. Then, for any $* \in  \{ alg, ab, h\} $,
\[  \Sp_*(\prod_{1 \leq i \leq n} S_i) \subset \bigcup_{1 \leq i \leq n} p_i^{-1} \left(\Sp_*(S_i) \right).\]
\end{prop}

\begin{proof}
The proofs for $* \in  \{ ab, h\} $ are straightforward. Let us prove the case $* = alg$. By induction, the general case follows from the case where $n = 2$. Let $X \subset S_1 \times S_2$ be a closed integral subvariety. Assume that $X \not \subset p_1^{-1}\left(\Sp_{alg}(S_1) \right)$, so that $Y := p_1(X)$ is of general type or has dimension $0$. If moreover $X \not \subset p_2^{-1}\left(\Sp_{alg}(S_2) \right)$, then the morphism $X \to Y$ admits at least one fiber of general type. Therefore, thanks to Theorem \ref{Nakayama} below (see also \cite[Theorem 3.9]{Bruni-relative-Lang}), the general fiber of the Stein factorization of $X \to Y$ is of general type. Since $Y$ is of general type, we get that $X$ is of general type, cf. \cite[Corollary IV]{Viehweg83}).
\end{proof}

\begin{thm}[Nakayama, {\cite[Theorem VI.4.3]{Nakayama}}]\label{Nakayama}
Let $X \to S$ be a projective surjective morphism with connected fibres from a normal complex analytic variety onto a smooth curve and $0 \in S$. Let $X_0= \cup_{i\in I} \Gamma_i$ the decomposition into irreducible components. If there is at least one irreducible component $\Gamma_j$ which is of general type (i.e. its desingularisation is such), then, for $s \in S$ general, the fiber $X_s$ is of general type.
\end{thm}

\begin{prop}\label{lemma-special-sets}
Let $I$ be a finite set and, for $i \in I$, let $q_i \colon X \to S_i$ be a surjective morphism between projective irreducible complex algebraic varieties. Suppose that the induced morphism $q \colon X \to \prod_{i \in I} S_i$ is finite and let $* \in  \{alg, ab, h\} $. Then,
\[ \forall i \in I,  \overline{\Sp_*(S_i)} \neq S_i \implies \overline{\Sp_*(X)} \neq X.\]
\end{prop}

\begin{proof}
It follows from Proposition \ref{Special sets and fibrations} and Proposition \ref{product-special-sets} that
\[ \Sp_\ast(X) \subset  q^{-1} \left( \Sp_*(\prod_{i \in I} S_i) \right) \subset \bigcup_{i \in I} q_i^{-1} \left(\Sp_*(S_i) \right),\]
so that $\overline{\Sp_\ast(X)} \subset \bigcup_{i \in I} \, q_i^{-1} (\overline{\Sp_\ast(S_i)} )$.
Since the $q_i$'s are surjective, it follows that $\overline{\Sp_\ast(X)} \neq X$.
\end{proof}

\section{Generalities on large local systems}\label{large local systems}

\subsection{Monodromy groups}

Let $k$ be a field. Let $\cL$ be a local system in $k$-vector spaces on a connected complex analytic space $X$. For $x \in X$, let $\rho_x \colon \pi_1(X, x) \to \GL(\cL_x)$ be the corresponding monodromy representation. The monodromy group of $\cL$ (with respect to the base-point $x$) is by definition the image of $\rho_x$; its Zariski-closure in $\GL(\cL_x)$ is called the algebraic monodromy group. Different points in $X$ yield isomorphic groups, respectively $k$-algebraic groups.

\subsection{Equivalent definitions of largeness}

\begin{defn}
A local system $\cL$ on a proper complex algebraic variety $X$ is called large if for every irreducible closed subvariety $Z \hookrightarrow X$, the pull-back of $\cL$ to the normalization of $Z$ is not isotrivial.
\end{defn}

If $f \colon X \to Y$ is a dominant morphism between two irreducible normal complex algebraic varieties, the image of the induced morphism of groups $f_\ast \colon \pi_1(X) \to \pi_1(Y)$ has finite index in $\pi_1(Y)$ \cite{Campana91}. Therefore, a local system $\cL$ on a proper complex algebraic variety $X$ is large if and only if for any non-constant morphism $f \colon Y \to X$ from an irreducible proper complex algebraic variety $Y$ the local system $f^{-1} \cL$ is not isotrivial. With this observation, the following results are immediate.

\begin{prop}\label{pull-back large}
Let $f \colon Y \to X$ be a finite morphism between proper varieties and $\cL$ a large local system on $X$. Then, the local system $f^{-1} \cL$ on $Y$ is large.
\end{prop}

\begin{prop}\label{large local systems and finite étale cover}
Let $f \colon Y \to X$ be a finite étale morphism between proper complex algebraic varieties. Let $\cL$ be a complex local system on $X$. Then, $\cL$ is large if and only if the pull-back local system $f^{-1} \cL$ on $Y$ is large.
\end{prop}

The following result is useful to prove that a local system is large.

\begin{prop}[compare with {\cite[Proposition 2.12]{Kollar-Shafarevich}}]\label{covering associated to a large local system}
Consider a local system $\cL$ on a connected proper complex algebraic variety $X$, and denote by $\tilde{X}^\cL \to X$ the associated connected covering space. Then the local system $\cL$ is large if and only if the complex analytic space $\tilde{X}^\cL$ does not contain any positive dimensional compact complex subspaces.
\end{prop}
\begin{proof} This is essentially the same proof as in {\cite[Proposition 2.12]{Kollar-Shafarevich}}. Let $Y \subset \tilde{X}^\cL$ be an irreducible compact complex subspace with normalization $\bar Y$. The induced holomorphic map $\bar Y \to X$ has discrete (hence finite) fibres. Let $Z \subset X$ be the image of $Y$ and $\bar Z$ denote its normalization. Then the monodromy group of $\cL_{|\bar Z}$ is isomorphic to the Galois group of $\bar Y \slash \bar Z$, in particular it is finite. This show that (up to deck transformations of $ \tilde{X}^\cL \to X$) there is a one-to-one correspondence between irreducible compact complex subspaces of $ \tilde{X}^\cL$ and irreducible  compact complex subspaces $Z \subset X$ such that the pull-back of $\cL$ to the normalization $\bar Z$ is isotrivial.
\end{proof}

\subsection{Shafarevich morphisms}

We will use in several places the existence of the Shafarevich morphism associated to a semisimple complex local system.
\begin{thm}\label{existence of Shafarevich morphism}
Let $\cL$ be a complex local system on a proper irreducible normal complex algebraic variety $X$. Assume that $\cL$ is semisimple, or equivalently that the algebraic monodromy group of $\cL$ is reductive.\\
Then, there exist a projective normal complex algebraic variety $\Sh_X^\cL$ and a fibration $\sh_X^\cL \colon X \rightarrow \Sh_X^\cL$, unique up to a unique isomorphism, with the following property: for any connected proper complex algebraic variety $Z$ and any morphism $f \colon Z \rightarrow X$, the composite map $\sh_X^\cL \circ f \colon Z \rightarrow \Sh_X^\cL$ is constant if and only if the local system
$f^{-1} \cL$ is isotrivial. \\
Moreover, if the monodromy group of $\cL$ is torsion-free, then there exists a (necessarily large) complex local system $\cM$ on $\Sh_X^\cL$ such that $\cL = \left(\sh_X^\cL \right)^{-1} \cM$.
\end{thm}

The map $\sh_X^\cL : X \rightarrow \Sh_X^\cL$ is called the Shafarevich morphism associated to $\cL$. Its existence for any semisimple complex local system $\cL$ is proved in \cite{Eyssidieux-reductive} for $X$ smooth projective (see also \cite{Brunebarbe-Shafarevich}). The existence of the Shafarevich morphism when $X$ is only normal proper is proved by applying the following result to a projective desingularization of $X$. The last part of Theorem \ref{existence of Shafarevich morphism} is \cite[Theorem 10.2]{Brunebarbe-Shafarevich}. 

\begin{prop}\label{Shafarevich: from smooth to normal}
Let $X^\prime \to X$ be a fibration between normal proper complex algebraic varieties. Let $\cL$ be a semisimple complex local system on $X$ and $\sh_{X^\prime}^{\cL^\prime} \colon X^\prime \rightarrow \Sh_{X^\prime}^{\cL^\prime}$ be the Shafarevich morphism associated to the (necessarily semisimple) local system $ \cL^\prime := \nu^{-1} \cL$. Then there is a (unique) factorization
\begin{align*}
\xymatrix{
X^\prime \ar[r] \ar[d] &  \Sh_{X^\prime}^{\cL^\prime} \\
X   \ar[ru]       &   }
\end{align*}
and the induced morphism $X \rightarrow \Sh_{X^\prime}^{\cL^\prime}$ is the Shafarevich morphism associated to $\cL$.
\end{prop}
\begin{proof}

Let $F$ be an irreducible component of the normalization of a fiber of $\nu$. Since the induced morphism $F \to X$ is constant, the restriction of $\cL$ to $F$ is trivial, therefore $F$ is mapped to a point by the composite map $F \to X^\prime \rightarrow \Sh_{X^\prime}^{\cL^\prime}$. Since the fibers of $\nu$ are connected and $X$ is normal, this shows that $\sh_{X^\prime}^{\cL^\prime} \colon X^\prime \rightarrow \Sh_{X^\prime}^{\cL^\prime}$ factors through a map $X \rightarrow \Sh_{X^\prime}^{\cL^\prime}$. The easy verification that this is the Shafarevich morphism associated to $\cL$ is left to the reader.
\end{proof}

\subsection{Variations of Hodge structure with discrete monodromy}

We recall the construction of the Shafarevich morphism associated to a complex local system that underlies a complex variation of Hodge structure with discrete monodromy.

\begin{prop}\label{Shafarevich map for PVHS}
Let $X$ be a connected normal projective complex algebraic variety and $\cL$ a complex local system on $X$. Assume that $\cL$ underlies a polarized complex variation of pure Hodge structure $(\cL, \cF^{\sbt}, h)$. Assume moreover that the monodromy group $\Gamma$ of $\cL$ is discrete in $\GL(\cL_x)$, so that the associated period map induces a holomorphic map $X \to \Gamma \backslash \cD$. Then, the analytification of the Shafarevich morphism associated to $\cL$ coincide with the Stein factorization of the proper holomorphic map $X \to \Gamma \backslash \cD$.
\end{prop}

\begin{proof}
See for example \cite[Proposition 3.5]{CCE} and its proof. Besides the algebraicity of the image of the period map, the key point is that every horizontal holomorphic map $W \to \cD$ from a compact complex manifold $W$ is constant, cf. \cite[Corollary 8.3]{Griffiths-Schmid}. 
\end{proof}

\begin{cor}
Same assumptions as in Proposition \ref{Shafarevich map for PVHS}. Then, the complex local system $\cL$ is large if and only if the period map $X \to \Gamma \backslash \cD$ associated to $(\cL, \cF, h)$ is finite.
\end{cor}

\begin{prop}\label{Special sets with PVHS}
Let $X$ be a normal proper complex variety and $\cL$ a large complex local system. Assume that $\cL$ underlies a polarized complex variation of pure Hodge structure with discrete monodromy. Then, 
\[ \Sp_{alg}(X) =  \Sp_{ab}(X) =  \Sp_{h}(X) = \varnothing. \]
\end{prop}
\begin{proof}
Thanks to \cite[Corollary 9.4]{Griffiths-Schmid}, every horizontal holomorphic map $\bC \to \cD$ is constant. This implies that $ \Sp_{h}(X) = \varnothing$. A fortiori, $ \Sp_{ab}(X) = \varnothing$ thanks to Lemma \ref{trivial inclusions}. Finally, $\Sp_{alg}(X) = \varnothing$ is a reformulation of the fact that any smooth proper complex variety that admits a polarized complex variation of pure Hodge structure with discrete monodromy and a generically finite period map is of general type, see e.g. \cite[Proposition 3.5]{CCE}.
\end{proof}

\subsection{Local systems with solvable monodromy}

\begin{prop}\label{Shafarevich in the solvable case}
Let $\cL$ be a large complex local system on a normal projective complex algebraic variety $X$. Let $\pi_1(X) \to \Gamma \subset \GL(n , \bC)$ be its monodromy representation. Assume that its monodromy group $\Gamma$ is solvable and the commutator subgroup $[\Gamma, \Gamma]$ is nilpotent. Then, $X$ admits a finite morphism to an abelian variety.
\end{prop}

\begin{proof}
Let us prove that the Albanese morphism $X \to \Alb(X)$ is finite. Let $f \colon Y \to X$ be a morphism from a smooth projective variety $Y$ such that the composite map $Y \to X \to \Alb(X)$ is constant, and let us prove that $f(Y)$ is necessarily a point. By assumption, the induced $\bQ$-linear map $f_\ast \colon \HH_1(Y, \bQ) \to \HH_1(X, \bQ)$ is zero, hence, up to replacing $Y$ by a finite étale cover, one can assume that the induced $\bZ$-linear map $f_\ast \colon \HH_1(Y, \bZ) \to \HH_1(X, \bZ)$ is zero. It follows that the canonical map $\HH_1(Y, \bZ) \to \Gamma^{\mathrm{ab}}$ is zero; equivalently the image of $\pi_1(Y)$ is contained in $[\Gamma, \Gamma]$ and therefore it is nilpotent. Applying the lemma below to the morphism $f$, we conclude that the image of $\pi_1(Y)$ in $\GL(n , \bC)$ is finite. Since $\cL$ is large, $f(Y)$ is necessarily a point.
\end{proof}

\begin{lem}[{see \cite[Lemme 4.6]{CCE} and the references therein}]\label{central serie}
Let $f \colon X \to Y$ be a morphism between two normal projective complex algebraic varieties. If the induced $\bQ$-linear map $f_\ast \colon \HH_1(X, \bQ) \to \HH_1(Y, \bQ)$ is zero, then the image of $\pi_1(X)$ in $\pi_1(Y) \slash C^k \pi_1(Y)$ is finite for every positive integer $k$. 
\end{lem}
Here, for every group $G$, $\{C^kG \}_{k \geq 0}$ denote the descending central series of $G$, defined by $C^0G = G$ and $C^{k+1} = [G, C^k G]$ for every integer $k \geq 0$. \\

The condition in Proposition \ref{Shafarevich in the solvable case} that $[\Gamma , \Gamma]$ is nilpotent is true up to replacing $X$ by a finite étale cover defined from a finite index subgroup of $\Gamma$ thanks to the following observation.

\begin{prop}\label{criterion derived group nilpotent}
If $\Gamma$ is a subgroup of $ \GL(n , \bC)$ contained in a connected solvable algebraic subgroup of $ \GL(n , \bC)$, then
$[\Gamma, \Gamma]$ is nilpotent.
\end{prop}
\begin{proof}
Let $\rG$ denote a connected solvable algebraic subgroup of $ \GL(n , \bC)$ that contains $\Gamma$. Its derived group $[\rG , \rG ]$ is a connected unipotent algebraic group, therefore $[\Gamma, \Gamma] \subset [\rG, \rG]$ is nilpotent.
\end{proof}



\subsection{A canonical decomposition}

\begin{prop}\label{fibration solvable/semisimple}
Let $X$ be a normal irreducible proper complex algebraic variety supporting a complex local system $\cL$. Then, there exist a finite étale cover $X^\prime \to X$ and a fibration $f \colon X^\prime \to Y$ onto a normal irreducible projective complex algebraic variety $Y$ such that:
\begin{itemize}
\item $Y$ admits a large complex local system with torsion-free monodromy and a semisimple algebraic monodromy group;
\item the monodromy of the restriction of $\cL$ to any fiber of $f$ is solvable.
\end{itemize} 
Moreover, if $\cL$ is large (resp. big), then the normalization of every fiber of $f$ admits a finite morphism to an abelian variety (resp. every generic fiber of $f$ admits a morphism to an abelian variety which is generically finite onto its image).
\end{prop}
\begin{proof}
Consider the monodromy representation $\pi_1(X) \to \rG(\bC)$ of $\cL$,  where $\rG$ is the algebraic monodromy group of $\cL$. Up to replacing $X$ with a finite étale cover, one can assume that $\rG$ is connected. Let $\rN$ be the solvable radical of $\rG$, so that $\rN$ is a connected solvable complex algebraic group. Let $\rH$ be the quotient of $\rG$ by $\rN$, so that $\rH$ is a connected semisimple complex algebraic group. The induced representation $\pi_1(X) \to \rH(\bC)$ has a Zariski-dense image, and replacing $X$ with a finite étale cover, one can assume that it has torsion-free image thanks to Selberg lemma. \\
Let $f \colon X \to Y$ denote the Shafarevich morphism associated to $\pi_1(X) \to \rH(\bC)$, cf. Theorem \ref{existence of Shafarevich morphism}. In particular, $Y$ is a normal irreducible projective complex algebraic variety and $f$ is surjective with connected fibres.
Since the representation $\pi_1(X) \to \rH(\bC)$ has torsion-free image, it factors through the homomorphism $\pi_1(X) \to \pi_1(Y)$, cf. Theorem \ref{existence of Shafarevich morphism}. The induced homomorphism $\pi_1(Y) \to \rH(\bC)$ corresponds to a large complex local system with a semisimple algebraic monodromy group.\\
Let $F$ be a (necessarily connected) fiber of $f$. Since the induced morphism $F \to Y$ is constant, the composite homomorphism $\pi_1(F) \to \pi_1(X) \to \rH(\bC)$ has finite image. But the image of $\pi_1(X) \to \rH(\bC)$ is torsion-free, hence the image of $\pi_1(F) \to \rH(\bC)$ is in fact trivial. Therefore it is contained in $\rN(\bC)$, from what it follows that it is solvable.\\
Assume now in addition that $\cL$ is large, and let $F^\prime$ be the normalization of an irreducible component of a fiber of $f$. Since the monodromy of the restriction of $\cL$ to $F^\prime$ is contained in the connected solvable algebraic group $\rN(\bC)$, one can apply Proposition \ref{Shafarevich: from smooth to normal}, Proposition \ref{Shafarevich in the solvable case} and Proposition \ref{criterion derived group nilpotent} to infer that $F^\prime$ admits a finite morphism to an abelian variety. The statement follows by taking a product over the irreducible components. The case where $\cL$ is big is similar and left to the reader.
\end{proof}


\subsection{Algebraic varieties with a large complex local system.}
The following statement collects some known results on algebraic varieties supporting a large complex local system.

\begin{thm}\label{known results}
Let $X$ be a connected normal projective complex algebraic variety with a large complex local system $\cL$.
\begin{enumerate}
\item Assume that the monodromy group of $\cL$ is solvable (equivalently the algebraic monodromy group of $\cL$ is solvable). Then, up to a finite étale cover, $X$ is isomorphic to the product of an abelian variety by a variety of general type.

\item Assume that $X$ is Brody-special\footnote{A proper complex algebraic variety $X$ is called Brody-special if there exists a Zariski-dense entire curve $\bC \to X$.}. Then, up to a finite étale cover, $X$ is isomorphic to an abelian variety.

\item Assume that the algebraic monodromy group of $\cL$ is semisimple. Then, $X$ is of general type.

\item Assume that $X$ is weakly-special\footnote{Following Campana \cite{Campana04}, a proper complex algebraic variety $X$ is called weakly special if it does not admit a finite étale cover $X'$ with a rational dominant map $X' \dasharrow Y$ to a positive dimensional variety of general type $Y$.}. Then, up to a finite étale cover, $X$ is isomorphic to an abelian variety.

\end{enumerate}
\end{thm}

\begin{proof}
For the first item, up to replacing $X$ by a finite étale cover, one can assume that the derived group of the monodromy group is nilpotent, cf. Proposition \ref{criterion derived group nilpotent}. It follows from Proposition \ref{Shafarevich in the solvable case} that $A$ admits a finite morphism to an abelian variety. Therefore, thanks to a result of Kawamata \cite[Theorem 13]{Kawamata-characterization}, after passing to another finite étale cover, X is biholomorphic to a product $B \times X^\prime$ of an abelian variety $B$ and a projective variety of general type $X^\prime$ whose dimension is equal to the Kodaira dimension $\kappa(X)$ of $X$. The second item is due to Yamanoi, see \cite{Yamanoi-Fourier} and \cite[Theorem 2.17]{Yamanoi-lectures}. The third item is due to Zuo \cite{Zuo-Chern-hyperbolicity}, see also \cite[Théorème 6.3]{CCE} for an alternative proof. For the last item, it follows from Proposition \ref{fibration solvable/semisimple} that, up to replacing $X$ with a finite étale cover, there exists a fibration $f \colon X \to Y$ onto a normal irreducible projective complex algebraic variety $Y$ such that:
\begin{itemize}
\item $Y$ admits a large complex local system with a semisimple algebraic monodromy group, and
\item the monodromy of the restriction of $\cL$ to the normalization of any fiber of $f$ is solvable.
\end{itemize} 
Thanks to the third item, $Y$ is of general type. Since $X$ is weakly-special, it follows that $Y$ is a point and that $\cL$ has solvable monodromy. But then the result follows from the first item.
\end{proof}

\section{The special subsets coincide}

\begin{thm}\label{Special sets of not general type varieties}
Let $X$ be a projective complex algebraic variety supporting a large complex local system $\cL$. If $X$ is not of general type, then $ \Sp_{alg}(X) = \Sp_{ab}(X) = \Sp_h(X) = X$.
\end{thm}

\begin{proof}
Thanks to Lemma \ref{trivial inclusions}, it is sufficient to prove that $\Sp_{ab}(X) = X$. It is harmless to assume that $X$ is irreducible. Note also that one can freely replace $X$ with any projective complex algebraic variety $X^\prime$ not of general type and such that there exists a finite surjective morphism $X^\prime \to X$. Indeed, the pull-back to $X^\prime$ of the local system $\cL$ is still large thanks to Proposition \ref{pull-back large}, whereas $\Sp_{ab}(X^\prime) = X^\prime$ implies $\Sp_{ab}(X) = X$. In particular, one can assume that $X$ is normal. We will also freely replace $X$ by any finite étale cover.\\
Up to replacing $X$ with a finite étale cover, one can assume that there exists a fibration $f \colon X \to Y$ as in Proposition \ref{fibration solvable/semisimple}. The normal projective variety $Y$ admits a large complex local system with a semisimple algebraic group, hence it is of general type thanks to Theorem \ref{known results}. Since by assumption $X$ is not of general type, the (geometric) generic fibre of $f$ is a positive-dimensional variety which is not of general type.\\
Let $C$ be an irreducible component of a fibre of $f$. Thanks to Theorem \ref{Nakayama}, $C$ is not of general type. On the other hand, by definition of $f$, the restriction of $\cL$ to the normalization of $C$ is a large complex local system with solvable monodromy. Therefore, thanks to Theorem \ref{known results}, the normalization of $C$ is, up to a finite étale cover, a product of a positive dimensional abelian variety by a variety of general type. This proves that the fibers of $f$ are covered by images of abelian varieties by finite maps, hence a fortiori $\Sp_{ab}(X) = X$.
\end{proof}

\begin{cor}\label{Equality of special subsets}
Let $X$ be a projective complex algebraic variety supporting a large complex local system $\cL$. Then,
\[  \Sp_{alg}(X) = \Sp_{ab}(X) = \Sp_h(X) .\]
\end{cor}

\begin{proof}
It follows from Theorem \ref{Special sets of not general type varieties} that $\Sp_{alg}(X) \subset  \Sp_{ab}(X)$, hence the equality thanks to Lemma \ref{trivial inclusions}. Moreover, given an entire curve $\bC \to X$, the normalization $Z$ of the Zariski-closure of its image in $X$ is connected and Brody-special. Since the induced morphism $Z \to X$ is finite, the pull-back of $\cL$ to $Z$ is still a large complex local system (cf. Proposition \ref{pull-back large}). Therefore a finite étale cover of $Z$ is isomorphic to an abelian variety thanks to Theorem \ref{known results}. This proves the inclusion $\Sp_{h}(X) \subset \Sp_{ab}(X)$, hence the equality thanks to Lemma \ref{trivial inclusions}. 
\end{proof}

\section{A non-Archimedean detour}

Throughout this section, let $k$ be a non-Archimedean local field (complete and locally compact by definition). Concretely, such a field $k$ is a finite extension either of $\bQ_p$ for some prime number $p$ or of a field of formal Laurent series $\mathbb{F}_q((T))$ over a finite field. The goal of this section is to prove the following result, which is a key step in the proof of Theorem \ref{Lang for big}.

\begin{thm}\label{non-Archimedean key}
Let $\rG$ be an absolutely almost-simple\footnote{An algebraic group over field $k$ is almost-simple if it is semisimple, noncommutative, and every proper normal subgroup is finite. It is absolutely almost-simple if its base-change to an algebraic closure of $k$ is almost-simple} algebraic group over $k$. Let $X$ be a connected smooth projective complex algebraic variety. If there exists a big representation $\rho \colon \pi_1(X) \rightarrow \rG(k)$ whose image is Zariski-dense in $\rG$ and unbounded in $\rG(k)$, then $\overline{\Sp_{ab}(X)} \neq X$.
\end{thm}

It is known by \cite[Theorem 1]{Zuo-Chern-hyperbolicity} and \cite[Theorem 6.3]{CCE} that a variety $X$ satisfying the hypotheses of Theorem~\ref{non-Archimedean key} is of general type. Our proof of Theorem~\ref{non-Archimedean key} relies on similar techniques, but it diverges on a key point: instead of using spectral covers (whose behavior under pull-back is unclear), we utilize a construction by Klingler that we recall in section \ref{A construction of Klingler} below.

\subsection{Katzarkov-Zuo reductions.}
We will use the Katzarkov-Zuo reduction of a $p$-adic representation, which is due to Eyssidieux \cite[Proposition 1.4.7]{Eyssidieux-reductive}, based on former works of Katzarkov and Zuo \cite{Katzarkov, Zuo-Chern-hyperbolicity, Zuo-book}.

\begin{thm}
Let $k$ be a non-Archimedean local field. Let $\rG$ be a reductive algebraic group over $k$. Let $X$ be a connected normal projective complex algebraic variety and $\rho \colon  \pi_1(X) \rightarrow \rG(k)$ a representation with Zariski-dense image. Then there exists a connected normal projective complex algebraic variety $S$ and a surjective algebraic map with connected fibres $\sigma \colon X \rightarrow S$ such that the following property holds: for any connected normal projective complex algebraic variety $Z$ and any algebraic map $f \colon Z \rightarrow X$, the composite map $\sigma \circ f \colon Z \rightarrow S$ is constant if and only if the representation $f^\ast \rho$ has bounded image.
\end{thm}

Observe that a fibration $\sigma \colon X \rightarrow S$ with this property is unique, up to unique isomorphism. It is called the Katzarkov-Zuo reduction of $(X,\rho)$. Its existence is proved in \cite{Eyssidieux-reductive} for a smooth $X$. However, one can argue as in Proposition \ref{Shafarevich: from smooth to normal} to prove its existence more generally when $X$ is normal.\\

The following result shows that in the situation of Theorem \ref{non-Archimedean key} the Katzarkov-Zuo reduction is a birational model of the Shafarevich morphism.

\begin{thm}\label{relating_KZ_to_Shafa}
Let $\rG$ be an almost-simple algebraic group over $k$. Let $X$ be a connected normal projective complex algebraic variety. Let $\rho \colon \pi_1(X) \rightarrow \rG(k)$ be a big representation whose image is Zariski-dense in $\rG$ and unbounded in $\rG(k)$. Then the Katzarkov-Zuo reduction $\sigma \colon X \rightarrow S$ of $(X, \rho)$ is birational.
\end{thm}

\begin{proof}
We can assume without loss of generality that $X$ is smooth. Let $X_s$ be a general fiber of the Katzarkov-Zuo reduction $\sigma \colon X \rightarrow S$ of $(X, \rho)$. The image of the homomorphism $\pi_1(X_s) \to \pi_1(X)$ is a normal subgroup of $\pi_1(X)$. Let $\Gamma$ denote the image of the representation $\rho \colon \pi_1(X) \to \rG(k)$. Let $\Delta$ denote the image of the composite map $\pi_1(X_s) \to \pi_1(X) \to \rG(k)$. The assumptions of Lemma \ref{Key_Lemma} below are fulfilled, hence $\Delta$ is finite. Since the representation $\rho$ is big, this is only possible if $X_s$ has dimension zero. This proves that $\sigma$ is birational.
\end{proof}

\begin{lem}\label{Key_Lemma}
Let $\rG$ be an almost-simple algebraic group over $k$. Let $\Gamma \subset \rG(k)$ be a Zariski-dense and unbounded subgroup. Let $\Delta$ be a normal subgroup of $\Gamma$. If $\Delta$ is bounded in $\rG(k)$, then $\Delta$ is finite.
\end{lem}
\begin{proof}
Let $\mathscr{B}(\rG, k)$ be the Bruhat-Tits building of $\rG$. Then $\mathscr{B}(\rG, k)$ can be realized as an open subset of a compact Hausdorff topological space $\overline{{\mathscr{B}}}(\rG, k)$, such that the action of $G(k)$ on $\mathscr{B}(\rG, k)$ extends to a continuous action of $G(k)$ on $\overline{{\mathscr{B}}}(\rG, k)$ and the stabilizer of every point in the boundary $\partial\overline{{\mathscr{B}}}(\rG, k) = \overline{{\mathscr{B}}}(\rG, k) \backslash \mathscr{B}(\rG, k)$ is contained in the $k$-points of a proper parabolic subgroup of $\rG$. For example, one may take $\overline{\mathscr{B}}(G, k)$ to be the compactification constructed by Borel and Serre in \cite{Borel-Serre}; see in particular Theorem 5.4 therein. In particular, a subgroup of $\rG(k)$ that fixes a point in the boundary $\partial\overline{{\mathscr{B}}}(\rG, k)$ cannot be Zariski-dense in $\rG$.\\

Since $\Delta$ is normal in $\Gamma$, its Zariski-closure $\bar{\Delta}^{Zar}$ is normal in $\bar{\Gamma}^{Zar} = \rG$. As $\rG$ is almost-simple, it follows that $\bar{\Delta}^{Zar}$ is either equal to $\rG$ or is a finite group. Let $\cF \subset \overline{{\mathscr{B}}}(\rG, k)$ be the subset of points that are fixed by the induced action of $\Delta$. This set is compact Hausdorff, since the action of $\rG(k)$ on $\overline{{\mathscr{B}}}(\rG, k)$ is continuous and $\overline{{\mathscr{B}}}(\rG, k)$ is compact Hausdorff. Moreover, $\cF$ is non-empty, as $\Delta$ is bounded by assumption. Since $\Delta$ is normal in $\Gamma$, the action of $\Gamma$ on $\overline{\mathscr{B}}(\rG, k)$ preserves $\cF$. If $\cF$ were entirely contained in the building $\mathscr{B}(\rG, k)$, then $\Gamma$ would fix the barycenter of $\cF$, contradicting the assumption that $\Gamma$ is unbounded. Hence, $\cF$ must intersect the boundary $\partial\overline{\mathscr{B}}(\rG, k)$. From the earlier discussion, this implies that $\bar{\Delta}^{\text{Zar}} \ne \rG$, and therefore must be finite. It follows that $\Delta$ itself is finite.
 \end{proof}

\begin{cor}\label{Katzarkov-Zuo images satisfy Lang}
Let $\rG$ be an absolutely almost-simple algebraic group over $k$. Let $X$ be a connected normal projective complex algebraic variety. Let $\rho \colon \pi_1(X) \rightarrow \rG(k)$ be a representation. Assume that the image of $\rho$ is torsion-free, Zariski-dense in $\rG$ and unbounded in $\rG(k)$. Let $\sigma \colon X \to S$ be the Katzarkov-Zuo reduction of $(X,\rho)$. Then $S$ is of general type and $\overline{\Sp_{ab}(S)} \neq S$.   
\end{cor}

\begin{proof}
We may harmlessly replace $X$ with a smooth projective variety mapping birationally to it. Therefore, thanks to \cite[Lemma 2.7]{CCE}, one may assume there exist a projective normal complex algebraic variety $Y$, a fibration $f \colon X \to Y$ and a big representation $\rho_Y \colon \pi_1(Y) \rightarrow \rG(k)$ on $Y$ such that $\rho = f^{-1} \rho_Y$. In particular, the map $\sigma \colon X \to S$ factors through $f \colon X \to Y$ and the induced map $\sigma_Y \colon Y \to S$ is the Katzarkov-Zuo reduction of $(Y,\rho_Y)$. Since $\rho$ and $\rho_Y$ have the same image, Theorem \ref{relating_KZ_to_Shafa} implies that the map $\sigma_Y \colon Y \to S$ is birational. Therefore, the result follows from Theorem \ref{non-Archimedean key} and Proposition \ref{Special sets and birational morphism}.
\end{proof}

\subsection{Klingler's local system}\label{A construction of Klingler}
In this section, we revisit Klingler’s construction \cite[Section 2.2.2]{Klingler2003} of a local system associated with a pluriharmonic map into a building.

\subsubsection{Recollections on buildings}
Let $\rG$ be a connected split semisimple algebraic group over $k$. Let $\rS$ be a split maximal torus of $\rG$. Let $X_\ast(S)$ be the group of $1$-parameter subgroups of $\rS$ and $X^\ast(S)$ be the group of characters of $\rS$. Let $ \langle \, ,  \rangle \colon X_\ast(S) \times X^\ast(S) \to \bZ$ be the perfect pairing of abelian groups such that $\langle \lambda, \chi \rangle$ is the integer such that $\left(\chi \circ \lambda \right) (t) = t^{\langle \lambda, \chi \rangle}$ for every $t \in \bG_m$. \\
Let $\rN$ be the normalizer of $S$ in $G$. The group $W := \rN(k) / \rS(k)$ is finite. The canonical action of $W$ on the real vector space $\cV := X_\ast(S) \otimes _\bZ \bR$ is faithful and identifies $W \subset \GL(\cV)$ with the Weyl group of the root system associated to $(\rG, \rS)$.\\
Denoting $\omega \colon k^\ast \to \bR$ the discrete valuation of $k$, there is a unique group homomorphism $\nu \colon \rS(k) \to \cV$ such that $ \langle \nu(z), \chi \rangle = - \omega (\chi(z)) $ for all $z \in \rS(k)$ and $\chi \in X^\ast(S)$. Let $S_c$ denote the kernel of $\nu$. Then $\Lambda := \rS(K) / S_c$ is a free abelian group of rank $\dim \rS = \dim_\bR \cV$, and the quotient $\tilde W := N(K)/ S_c$ is an extension of $W$ by $\Lambda$.\\
There is a real affine space $\cA = \cA(\rG, \rS, k)$ under $\cV$, unique up to unique isomorphism, such that $\nu$ extends to a homomorphism of $\tilde \nu \colon \rN(k) \to \mathrm{Aff}(\cA)$ in the group of affine transformations of $\cA$. The homomorphism $\tilde \nu$ factors through $\rN(k) \to \tilde W$, and the two composite maps $\tilde W \to \mathrm{Aff}(\cA) \to \GL(\cV)$ and $\tilde W \to W \subset \GL(\cV)$ are equal.\\
Let $\mathscr{B}(\rG, k)$ be the Bruhat-Tits building of $\rG / k$. Then $\cA$ is the real affine space on which the apartments of $\mathscr{B}(\rG, k)$ are modeled. It will be important to note for later that for any finite field extension $l \supset k$ (where the valuation on $l$ induces the valuation on $k$), the Weyl group $W$ and the $W$-module $\cV$ remain canonically identified. 

\subsubsection{Recollections on harmonic maps}

Let $X$ be a Riemannian manifold and $f \colon X \to \mathscr{B}(\rG, k)$ be a locally Lipschitz continuous map. A point $x \in X$ is called regular for $f$ if there is an apartment in $\mathscr{B}(\rG, k)$ that contains $f(U)$ for a sufficiently small neighborhood $U$ of $x$ in $X$ \cite[p.225]{Gromov-Schoen}. Otherwise $x$ is called singular. The map $f$ is called harmonic if for every point $x \in X$, there exists a small ball $B$ centered at $x$ on which $f$ minimizes the energy relatively to $f_{|\partial B}$ \cite[p.232]{Gromov-Schoen}. (We refer to \cite{Gromov-Schoen} for the definition of the energy.) \\

Assume now that $X$ is a connected compact K\"ahler manifold and let  $\rho \colon \pi_1(X) \rightarrow \rG(k)$ be a representation with Zariski-dense image. Thanks to \cite[Theorem 7.1]{Gromov-Schoen} there exists a Lipschitz harmonic $\rho$-equivariant map $f \colon \tilde{X} \to \mathscr{B}(\rG, k)$ from the universal covering of $X$ (with finite energy since $X$ is compact). The subset $\widetilde{\Reg}(f) \subset \tilde{X}$ of regular points for $f$ is a $\pi_1(X)$-invariant open subset of $\tilde{X}$, and one denotes by $\Reg(f)$ its image in $X$. The Hausdorff codimension of its complement $\Sing(f) \subset X$ is at least $2$ \cite[Theorem 6.4]{Gromov-Schoen}. Moreover $f$ is pluriharmonic, i.e. $\partial \bar \partial f = 0$ on $\widetilde{\Reg}(f)$\cite[Theorem 7.3]{Gromov-Schoen}, and for any holomorphic map $g\colon Y \to X$ from a connected compact K\"ahler manifold $Y$ with universal covering $\tilde{Y}$, the representation $g^\ast \rho \colon \pi_1(Y) \rightarrow \rG(k)$ is semisimple (i.e., the Zariski-closure of its image is a reductive group) \cite{Corlette_JDG} and the composite map $f \circ \tilde{g} \colon \tilde{Y} \to \mathscr{B}(\rG, k)$ is $g^\ast \rho $-equivariant and pluriharmonic \cite[Corollaire 1.3.8]{Eyssidieux-reductive}.

\subsubsection{Definition of Klingler's local system}
Klingler explains in \cite[section 2.2.2]{Klingler2003} the construction of a complex local system $F(X, \rho)$ on $\Reg(f)$ with finite monodromy that corresponds intuitively to the pull-back by $f$ of the complexified tangent bundle of the building $\mathscr{B}(\rG, k)$. We briefly recall the construction and refer to \textit{loc. cit.} for the details.\\

Let $x \in \widetilde{\Reg}(f)$, so that there exists an isometric embedding $i \colon \cA \subset \mathscr{B}(\rG, k)$ and a neighborhood $B$ of $x$ in $\tilde{X}$ such that the map $f_{|B} \colon B \to \mathscr{B}(\rG, k)$ factors through a pluriharmonic map $h \colon B \to \cA$. The map $h$ is well-defined up to the action of $\tilde W$ on $\cA$. Since $h$ is pluriharmonic, by taking the $(1,0)$-part of the complexification of its differential, one obtains a $\bC$-linear map $\cV_{\bC}^\vee \to \Omega^1_B$, well-defined up to the action of $W$ on $\cV_{\bC}^\vee$. More precisely, the germs of all maps $h$ and $dh^{(1,0)}$ form respectively a $\tilde W$-torsor and a $W$-torsor on $\widetilde{\Reg}(f)$, related by the homomorphism $\tilde W \to W$. These torsors descend to $\Reg(f)$. Therefore, we get from the map $f$ a canonical monodromy representation $\pi_1(\Reg(f)) \to \tilde W$, and by composing with the homomorphisms $\tilde W \to W$ and $W \subset \GL(\cV)$, we obtain a real local system $F_{\bR}(f)$ on $\Reg(f)$ corresponding to the monodromy representation $\pi_1(\Reg(f)) \to W \subset \GL(\cV)$. The derivative of $f$ yields a real one-form $\mu^{\bR}_f$ on $\Reg(f)$ with values in $F_{\bR}(f)$. We denote by $F(f)$ the complex local system associated to $F_{\bR}(f)$. Since $h$ is pluriharmonic, the complexification of $\mu^{\bR}_f$ is a holomorphic one-form $\mu_f$ on $\Reg(f)$ with values in $F(f)$. The one-form $\mu_f$ is everywhere zero if and only if $f$ is constant, if and only if $\rho$ has bounded image.\\

The following compatibilities follow readily from the construction: 
\begin{itemize}
    \item Let $Y$ be a connected compact K\"ahler manifold. Let $g \colon Y \to X$ be a holomorphic map with lifting $\tilde g \colon \tilde Y \to \tilde X$, so that the map $f \circ \tilde g$ is pluriharmonic. If $g(Y) \not \subset \Sing(f)$, then $F(f \circ \tilde g) = g^{-1} F(f)$ on $\Reg(f \circ \tilde g) \cap g^{-1} \left( \Reg(f) \right)$.
\item If $l \supset k$ is a finite extension and $\mathscr{B}(\rG, k) \subset \mathscr{B}(\rG, l)$ is the canonical $\rG(k)$-equivariant inclusion, then $f$ induces a pluriharmonic map $f_l \colon \tilde X \to \mathscr{B}(\rG, l)$ which is equivariant with respect to the representation $\pi_1(X) \to \rG(k) \subset \rG(l)$. Then $\Reg(f) \subset \Reg(f_l)$ and that $F(f_l)_{|\Reg(f)} = F(f)$.
\end{itemize}

\begin{prop}\label{Klocal system trivial on abelian varieties}
Let $\rG$, $\mathscr{B}(\rG, k)$, $X$, $\rho \colon \pi_1(X) \to \rG(k)$, $f \colon \tilde X \to \mathscr{B}(\rG, k)$ as above. Let $Y$ be a connected compact K\"ahler manifold with an abelian fundamental group. Let $g \colon Y \to X$ be a holomorphic map with lifting $\tilde g \colon \tilde Y \to \tilde X$. Let $\rho_{|Y} := g^{-1} \rho \colon \pi_1(Y) \to \rG(k)$. If $g(Y) \not \subset \Sing(f)$, then there exists a finite extension $l \supset k$ such that the image of the pluriharmonic map $f \circ \tilde g \colon \tilde Y \to  \mathscr{B}(\rG, k)$ is contained in an apartment of $\mathscr{B}(\rG, l)$ via the canonical $\rG(k)$-equivariant inclusion $\mathscr{B}(\rG, k) \subset \mathscr{B}(\rG, l)$. In particular, the local system $F(f \circ \tilde g)$ is trivial.   
\end{prop}

In fact the proof will give an explicit description of the map $f \circ \tilde g$.
\begin{proof}
We keep the notation introduced at the beginning of this section.\\
Let $H$ be the Zariski closure of the image of $\rho_Y$. It is a $k$-algebraic group, and up to replacing $Y$ with a finite étale cover, one can assume that $H$ is a connected $k$-algebraic group.  Recall that by a theorem of Corlette \cite{Corlette_JDG}, the pull-back of a semisimple complex local system  by an algebraic map between smooth complex projective varieties is again semisimple. Therefore, choosing an embedding of $k$ in $\bC$, we see that $H$ is a reductive $k$-algebraic group. Moreover, since by assumption $\pi_1(Y)$ is abelian, $H$ is commutative. It follows that $H$ is a torus.  \\
Up to replacing $k$ with a finite extension, one can assume that $H$ is split over $k$. We may also assume that the split maximal torus $S$ contains $H$. Let $h \colon \tilde Y \to \cA$ be a pluriharmonic map which is equivariant with respect to the representation obtained by composing $\pi_1(Y) \to H(k) \subset S(k)$ with $\nu \colon S(k) \to \cV \subset \mathrm{Aff}(\cA)$. Fix $y \in \tilde Y$ and choose an embedding $i \colon \cA \to  \mathscr{B}(\rG, k)$ identifying $\cA$ with an apartment of $\mathscr{B}(\rG, k)$ that contains $f(y)$. Up to translating $h$ with an element of $\cV$, one can assume that the two pluriharmonic maps $f \circ \tilde g$ and $i \circ h$ coincide at $y \in \tilde Y$.\\
If $d$ denote the Bruhat-Tits distance on $\mathscr{B}(\rG, k)$, the function $\tilde Y \to \bR, x \mapsto d(f \circ \tilde g(x),i \circ h(x))$ is $\rho_{|Y}$-equivariant and plurisubharmonic \cite[Lemma 5.3]{Gromov-Schoen}. Therefore, it is the pull-back of a plurisubharmonic function on the compact complex manifold $Y$, hence it is constant by the maximum principle. Since it is zero at $y$, it is zero everywhere. This proves that $f \circ \tilde g = i \circ h$. In particular, the image of the map $f \circ \tilde g \colon \tilde Y \to  \mathscr{B}(\rG, k)$ is contained in the apartment $i(\cA)$ of $\mathscr{B}(\rG, k)$. 
\end{proof}

\subsection{Proof of Theorem~\ref{non-Archimedean key}}

\begin{prop}\label{non-Archimedean key-special case}
Let $\rG$ be an almost-simple algebraic group over $k$. Let $X$ be a connected smooth projective variety. Let $\rho \colon \pi_1(X) \rightarrow \rG(k)$ be a big representation whose image is Zariski-dense in $\rG$ and unbounded in $\rG(k)$. Let $f \colon \tilde{X} \to \mathscr{B}(\rG, k)$ be a pluriharmonic $\rho$-equivariant map. Assume that the associated local system $F(f)$ is trivial. Then:
\begin{enumerate}
    \item The Albanese map $\alb_X \colon X \to \Alb(X)$ is generically finite.  
    \item $X$ is of general type.
\end{enumerate}
\end{prop}

In particular, it follows from a result of Yamanoi \cite[Corollary 1]{Yamanoi-maximal-Albanese} that $\overline{\Sp_{ab}(X)} \neq X$.

\begin{proof}
The subset $\Sing(f)$ is contained in a closed analytic subset of $X$ \cite[Proposition 1.3.3]{Eyssidieux-reductive} distinct from $X$. Since $f$ is Lipschitz, the holomorphic one-form $\mu_f \in \HH^0(\Reg(f) , \Omega^1 \otimes_\bC \cV)$ is bounded, hence it extends to a global holomorphic one-form $\mu_f \in \HH^0(X, \Omega^1_X \otimes_\bC \cV)$. Since the image of $\rho$ is unbounded, the map $f$ is nonconstant and $\mu_f \neq 0$. In particular, the map  $\alb_X$ is non constant. \\
Let $F$ be a connected component of a generic fiber of $\alb_X$. The restriction of $\mu_f$ to $F$ is zero, hence the restriction of $f$ to the preimage of $F$ in $\tilde X$ is constant. Equivalently, the pull-back of $\rho$ to $F$ has bounded image. However, thanks to Theorem~\ref{relating_KZ_to_Shafa}, the Katzarkov-Zuo reduction of $(X, \rho)$ is birational, hence $F$ is necessarily a point. Therefore, $\alb_X$ is generically finite.\\

Let us now prove that $X$ is of general type, following an argument essentially due to Zuo \cite{Zuo-Chern-hyperbolicity}. Since $\alb_X$ is generically finite, up to replacing $X$ by a birational model, the generic fiber of the Iitaka fibration of $X$ is an abelian variety $A$ \cite[Proposition 17.5.1]{Kollar-Shafa-book}; see also \cite[Theorem 13]{Kawamata-characterization}.
The image of the group homomorphism $\pi_1(A) \to \pi_1(X)$ is a normal subgroup of $\pi_1(X)$ and the image of $\pi_1(X)$ by $\rho$ is Zariski-dense in $\rG$. Therefore, the Zariski-closure $\rH$ of the image of $\pi_1(A)$ in $\rG$ is a normal algebraic subgroup of $\rG$. But $\pi_1(A)$ is commutative, hence $\rH$ is commutative too. Since $\rG$ is almost-simple, $\rH$ is necessarily a finite subgroup of $\rG$. Since $\rho$ is big, this forces $\dim A = 0$, hence $X$ is of general type.
\end{proof}

We now finish the proof of the Theorem~\ref{non-Archimedean key}. We keep the notation from the statement. Up to replacing $k$ by a finite extension, one may assume that $\rG$ is split over $k$. Let $f \colon \tilde{X} \to \mathscr{B}(\rG, k)$ be a pluriharmonic $\rho$-equivariant map. Since $\Sing(f)$ is contained in a closed analytic subset of $X$ \cite[Proposition 1.3.3]{Eyssidieux-reductive} distinct from $X$ and since the monodromy of the local system $F(f)$ is finite, there exists a normal ramified Galois covering $p \colon X^\prime \to X$, étale over $\Reg(f)$, such that the complex local system $p ^{-1} \left(F(f) \right)$ on $p^{-1}\left(\Reg(f) \right)$ is trivial, hence extends to $X^\prime$. Let $Z$ be a smooth projective variety and $Z \to X^\prime$ be a birational map which is an isomorphism in corestriction to $p^{-1}\left(\Reg(f) \right)$. Let $\pi \colon Z \to X$ denote the induced map, with lifting $\tilde \pi \colon \tilde Z \to \tilde X$. Since $\pi$ is finite surjective, the image of the homomorphism $\pi_\ast \colon \pi_1(Z) \to \pi_1(X)$ has finite index in $\pi_1(X)$, hence the image of the induced representation $\rho_{|Z} \colon \pi_1(Z) \to \rG(k)$ is also Zariski-dense and unbounded. Moreover, $\rho_{|Z}$ is also a big representation. Therefore, $Z$, $\rho_{|Z}$ and $f \circ \tilde{\pi}$ satisfies the assumption of Proposition~\ref{non-Archimedean key-special case}. Therefore $Z$ is a projective variety of general type with a generically finite map to an abelian variety. It follows by \cite[Corollary 1]{Yamanoi-maximal-Albanese} that $\overline{\Sp_{ab}(Z)} \neq Z$.\\

Let $a \colon A \dasharrow X$ be a non-constant rational map from an abelian variety. Assume that the image of $a$ is not contained in $\Sing(f)$. Let $Y$ be a projective desingularization of the graph of $a$, so that there is an induced algebraic map $g \colon Y \to X$, with lifting $\tilde g \colon \tilde Y \to \tilde X$. Since $Y$ is birational to $A$, its fundamental group is abelian. In particular, thanks to Proposition~\ref{Klocal system trivial on abelian varieties}, the complex local system $g^{-1} F(f)= F(f \circ \tilde g)$ is trivial.\\

Then $Y^o := g^{-1}(\Reg(f))$ is a dense Zariski-open subset of $Y$ contained in $\Reg(f \circ \tilde g)$, and $g$ induces a map $g_{|Y^o} \colon Y^o \to \Reg(f)$. By construction, the complex local systems $(g_{|Y^o})^{-1} F(f)$ and $F(f \circ \tilde g)_{|Y^o}$ are isomorphic. Since the local system $F(f \circ \tilde g)$ is trivial, there is an algebraic map $Y^o \to Z$ lifting $g_{|Y^o} \colon Y^o \to X$. It follows that $\Sp_{ab}(X)$ is contained in the union of $\Sing(f)$ and the image of $\Sp_{ab}(Z)$. Since $\overline{\Sp_{ab}(Z)} \neq Z$, we conclude that $\overline{\Sp_{ab}(X)} \neq X$.

 
\section{Proof of Theorem \ref{Lang for big}}
In this section we complete the proof of Theorem~\ref{Lang for big}.

\subsection{A structural result}
 As a first step, we establish a structural result, whose formulation and proof are inspired by the construction of the Shafarevich morphism in \cite{Eyssidieux-reductive} and \cite{Brunebarbe-Shafarevich}.
\begin{prop}\label{Structure result}
Let $\cL$ be a large complex local system on a connected normal projective complex algebraic variety $X$. Assume that the algebraic monodromy group $\rG$ of $\cL$ is a connected reductive complex algebraic group (the unique split reductive $\bZ$-algebraic group scheme whose base-change to $\bC$ is $\rG$ will also be denoted $\rG$ somewhat abusively). Then there exist finitely many representations $\rho_i$ such that:
\begin{enumerate}
\item Every $\rho_i$ is of the form $\pi_1(X) \rightarrow \rG(k_i)$, with $k_i$ a finite extension of $\bQ_p$ for some prime number $p$; the image of $\rho_i$ is Zariski-dense in $\rG$ and unbounded in $\rG(k_i)$;
\item If $X \to S_i$ denote the Katzarkov-Zuo reduction of $\rho_i$, then the normalization of any fiber  of the induced algebraic map $X \to \prod_i S_i$ admits a large complex local system underlying a polarized variation of pure Hodge structures with discrete monodromy.
\end{enumerate} 
Moreover, up to replacing $X$ by a finite étale cover, one can assume that the image of the representations $\rho_i$ are torsion-free.    
\end{prop}

\begin{proof} 
We first recall some preliminaries. Let $Y$ be a connected complex algebraic variety. Fix $y \in Y$. Let $R(Y,y, \rG)$ be the affine $\bQ$-scheme of finite type that represents the functor that associates to any $\bQ$-algebra $A$ the set of group homomorphisms from $\pi_1(Y,y)$ to $\rG(A)$. We denote by $M_B(Y, \rG)$ the affine $\bQ$-scheme of finite type corresponding to the finitely generated $\bQ$-algebra $\bQ[R(\pi_1(Y,y), \rG)]^{\rG}$. (For any other choice of base-point $y^\prime \in Y$, the $\bQ$-schemes $M_B(Y,y^\prime,G)$ and $M_B(Y,y,G)$ canonically isomorphic.) A representation $\rho \colon \pi_1(Y,y) \to \rG(\bC)$ is called reductive if the Zariski-closure of its image is a reductive algebraic subgroup of $\rG(\bC)$. The $\bC$-points of $M_B(Y, \rG)$ are naturally in bijection with the $\rG(\bC)$-conjugacy classes of elements in $R(Y,y, \rG)(\bC)$ corresponding to reductive representations. See \cite{Lubotzky-Magid} and \cite{Sikora} for details.\\

Assume that $Y$ is smooth projective. The Corlette-Simpson non-abelian Hodge correspondence yields an equivalence between the category of reductive representations $\pi_1(Y,y) \to \rG(\bC)$ and the category of polystable Higgs $\rG$-torsors with vanishing rational Chern classes. This endows the topological space $M_B(Y, \rG)(\bC)$ with a functorial continuous action of $\bC^\ast$ by scaling the Higgs field of the corresponding Higgs bundle. The fixed points of this action are precisely the $\rG(\bC)$-conjugacy classes of (necessarily reductive) representations $\pi_1(Y,y) \to \rG(\bC)$ that occur as the monodromy of a complex local system underlying a complex polarized variation of Hodge structures ($\bC$-VHS). See \cite{SimpsonHiggs}. \\

We proceed to the proof of the proposition, keeping the notation from its statement. Fix $x \in X$. By \cite[Proposition 8.2]{A'Campo-Burger}, the subset of $R(X,x, \rG)(\bC)$ consisting in representations with Zariski-dense image in $\rG$ is a dense Zariski-open subset defined over $\bQ$. Its image $M_B(X, \rG)(\bC)^{ZD}$ in $M_B(X, \rG)(\bC)$ is therefore a $\bQ$-constructible subset. \\

Let $\Sigma$ be the collection of all representations $\rho \colon \pi_1(X) \rightarrow \rG(k)$ with values in a finite extension $k$ of $\bQ_p$ for some prime $p$ and with Zariski-dense image in $\rG$. Let $\Theta \subset \Sigma$ be a finite subset and let $X \to S_\Theta =  \prod_{\rho \in \Theta}  S_\rho$ be the product of the Katzarkov-Zuo reductions $X \to S_\rho$ of all $\rho \in \Theta$.
The associated equivalence relation on $X$ is defined by the closed subvariety $X \times_{S_\Theta} X$ of $X \times X$. If $\Theta^\prime \subset \Sigma$ is another finite subset containing $\Theta$, then $X \times_{S_{\Theta^\prime}} X$ is contained in $X \times_{S_\Theta} X$. Therefore, by noetherianity, one can assume that $X \times_{S_{\Theta}} X = X \times_{S_{\Sigma}} X$. In other words, any representation $\rho \in \Sigma$ has bounded image after restriction to any irreducible component of a fiber of the induced algebraic map $X \to S_\Theta$. One may and will also assume that every $\rho \in \Theta$ has unbounded image in $\rG(k)$.\\

Let $F$ be the normalization of an irreducible component of a fiber of the map $X \to S_\Theta$. Let $M$ denote the image of $M_B(X, \rG)(\bC)$ in $M_B(F,\rG)(\bC)$. It is a $\bQ$-constructible subset of $M_B(F,\rG)(\bC)$. By applying \cite[Theorem 6.8]{Brunebarbe-Shafarevich} to $M_B(X, \rG)(\bC)^{ZD}$, it follows that $M$ consists in finitely many points. Moreover, if $\tilde F$ is a desingularization of $F$, then the image of $M_B(X, \rG)(\bC)$ in $M_B(\tilde F,\rG)(\bC)$ is $\bC^\ast$-invariant and consists a fortiori in finitely many points. Therefore, all these points are $\bC^\ast$-fixed, hence comes from a $\bC$-VHS. Since $F$ is normal, this implies that all elements in $M$ are themselves coming from $\bC$-VHS.\\

In particular, the pull-back of $\cL$ to $F$ underlies a $\bC$-VHS and, if 
\[ \rho^\cL \colon \pi_1(X,x) \to \rG(\bC)\]
denotes the monodromy representation of $\cL$ and $\rho^\cL_{|F}$ its pull-back to $F$, then $\rho^\cL_{|F}$ is $\rG(\bC)$-conjugated to a representation with values in $\rG(K)$ for some number field $K$. After replacing $K$ with a suitable finite extension, the representation $\rho^\cL_{|F}$ can also be viewed as the restriction of a representation $\pi_1(X,x) \to \rG(K)$ with a Zariski-dense image. The definition of $\Theta$ then implies that for every prime ideal $\mathcal{P}$ in the ring of integers $\mathcal{O}_K$ of $K$, the corresponding representation into the completion, $\pi_1(F) \to \rG(K_\mathcal{P})$, has a bounded image. Consequently, $\rho^\cL_{|F}$ must take values in $\rG(\mathcal{O}_K)$. \\

Since $M$ is defined over $\bQ$, the $\mathrm{Gal}(\bar \bQ / \bQ)$-conjugates of $\rho^\cL_{|F}$ also belong to $M$. Therefore, they come from a $\bC$-VHS and are $\rG(\bC)$-conjugated to a representation with values in $\rG(\bar \bZ)$. Taking their sum, we obtain a representation coming from a $\bC$-VHS with integral monodromy up to conjugation. Furthermore, the resulting representation is large, since at least one (and in fact all) of its direct summands is large.\\

Finally, we claim that, after possibly replacing $X$ by a finite étale cover, one may assume that the images of all representations in $\Theta$ are torsion-free. Let $\pi \colon X^\prime \to X$ be a finite étale cover such that all the representations $\pi^{-1} \rho$ with $\rho \in \Theta$ have a torsion-free image. For every $\rho \in \Theta$, the Katzarkov-Zuo reduction $X^\prime \to S_\rho^\prime$ of $\pi^{-1} \rho$ is the Stein factorization of the composite map $X^\prime \to X \to S_\rho$. It follows that the pull-back of $M_B(X, \rG)$ to the normalization of any irreducible component of a fiber of the induced algebraic map $X^\prime \to \prod_{\rho \in \Theta} S_\rho^\prime$ consists also in finitely many points. Therefore one can proceed as above.
\end{proof}

We turn to the proof of Theorem~\ref{Lang for big}. Let $X$ be a connected normal projective complex algebraic variety endowed with a big complex local system $\cL$. We shall prove the equivalence of the four properties stated in Theorem~\ref{Lang for big}. By Proposition~\ref{Special sets and finite étale cover}, we may, without loss of generality, replace $X$ by a finite étale cover. In particular, we may assume that the algebraic monodromy group $\rG$ of $\cL$ is connected.

\subsection{The semisimple case}
We first consider the case where the algebraic monodromy group $\rG$ of $\cL$ is semisimple. Then, by Theorem \ref{known results},  the algebraic variety $X$ is of general type. Therefore, in this special case, Theorem \ref{Lang for big} becomes equivalent to the non Zariski density of $\Sp_{\ast}(X)$ in $X$ for any $\ast \in \{ alg, ab, h \}$.\\

By Theorem \ref{existence of Shafarevich morphism}, up to replacing $X$ by a finite étale cover, there exists a normal projective variety $Y$, a large local system $\cM$ on $Y$ and fibration $f \colon X \to Y$ such that $f^{-1} \cM = \cL$. Since by assumption $\cL$ is big, $f$ is a birational map. Therefore, thanks to Proposition \ref{Special sets and birational morphism}, one may assume that $\cL$ is large.\\

Also, one may reduce to the case where $\rG$ is connected almost-simple simple as follows. Since $\rG$ is a connected semisimple complex algebraic group, there exists an isogeny $\rG \to \prod_i \rG_i$ onto a product of connected almost-simple complex algebraic groups. For every $i$, let $X \to X_i$ be the Shafarevich morphism of the induced (Zariski-dense) representation $\rho_i \colon \pi_1(X) \to \rG_i(\bC)$. Up to replacing $X$ by a finite étale cover, one can assume that $\rho_i$ factors through $X_i$ for every $i$. Since $\rho$ is large, the induced morphism $X \to \prod_i X_i$ is finite.
Therefore, if we know that $\Sp_{\ast}(X_i)$ is not Zariski-dense in $X_i$, then it follows by Proposition \ref{lemma-special-sets} that $\overline{\Sp_{\ast}(X)} \neq X$.\\

Therefore, it remains to consider the case where $\cL$ is large and $\rG$ is a connected almost-simple complex algebraic group. Thanks to Corollary~\ref{Equality of special subsets}, we need only to prove that $\Sp_{ab}(X)$ is not Zariski dense in $X$. With the notation of Proposition \ref{Structure result}, let $f\colon X \to Y$ denote the Stein factorization of the morphism $X \to \prod_i S_i$. Note that for every $i$ the induced morphism $Y \to S_i$ is surjective, since its precomposition $X \to S_i$ with the canonical morphism $X \to Y$ is surjective. Thanks to Corollary \ref{Katzarkov-Zuo images satisfy Lang}, $\overline{\Sp_{ab}(S_i)} \neq S_i$ for every $i \in I$. Using Proposition \ref{lemma-special-sets}, it follows that $\overline{\Sp_{ab}(Y)} \neq Y$. \\

Therefore, in view of Proposition \ref{Special sets and fibrations}, it is sufficient to prove that $\Sp_{ab}(X/Y)$ is not Zariski-dense in $X$. Since $X$ is normal, there exists a Zariski-dense open $Y^o$ of $Y$ over which the (geometric) fibers of $f$ are normal. Moreover, by Proposition \ref{Structure result}, any such fiber $X_y$ admits a large complex local system that underlies a polarized variation of pure Hodge structures with discrete monodromy. It follows from Proposition \ref{Special sets with PVHS} that $\Sp_{ab}(X_y) = \varnothing$ for any $y \in Y^o$. This proves that $\Sp_{ab}(X/Y)$ is contained in $f^{-1}(Y \backslash Y^o)$, hence it is not Zariski-dense in $X$. This concludes the proof of Theorem \ref{Lang for big} when $\rG$ is semisimple.


\subsection{The general case}
We now make no assumption on the algebraic monodromy group $\rG$ of $\cL$.
By Proposition \ref{fibration solvable/semisimple}, up to replacing $X$ by a finite étale cover, there exists a fibration $f \colon X \to Y$ onto a normal irreducible projective complex algebraic variety $Y$ such that:
\begin{itemize}
\item $Y$ admits a large complex local system with a semisimple algebraic monodromy group;
\item the generic fiber $F$ of $f$ admits a morphism to an abelian variety which is generically finite onto its image.
\end{itemize}
We know by Theorem \ref{known results} that $Y$ is of general type and we have proved in the preceding section that $\overline{\Sp_{\ast}(Y)} \neq Y$ for any $\ast \in \{ alg, ab, h \}$. Moreover, thanks to \cite{Yamanoi-maximal-Albanese} and \cite{Bruni-relative-Lang}, for any $\ast \in \{ alg, ab, h \}$, $F$ is of general type if, and only if, $\Sp_{\ast}(F)$ is Zariski-dense in $F$.\\

Assume that $X$ is not of general type. Then, by a result of Kollár \cite[p.363, Theorem]{Kollar-subadditivity}, $F$ is not of general type. Consequently, for any $\ast \in \{ alg, ab, h \}$, $\Sp_{\ast}(F)$ is Zariski-dense in $F$, and therefore $\Sp_{\ast}(X)$ is Zariski dense in $X$ as well.\\

Assume that $X$ is of general type. It follows from Iitaka’s easy addition formula \cite[Theorem 11.9]{Iitaka-book} (see also \cite[Lemma 2.3.31]{Fujino-book}) that $F$ is of general type. By applying the following result to the corestriction of $f$ to its smooth locus, we get that $\Sp_{\ast}(X/Y)$ is not Zariski-dense in $X$ for any $\ast \in \{ alg, ab, h \}$.

\begin{thm}[cf. \cite{Bruni-relative-Lang}]
Let $f \colon X \to Y$ be a smooth proper surjective morphism with connected fibers between smooth complex algebraic varieties. Assume that the fibers of $f$ over a Zariski-dense open subset of $Y$ are of general type and of maximal Albanese dimension. Then, for any $\ast \in \{ alg, ab, h \}$, $\Sp_{\ast}(X/Y)$ is not Zariski-dense in $X$. 
\end{thm}

Since by Proposition \ref{Special sets and fibrations} one has, for any $\ast \in \{ alg, ab, h \}$,
\[ \Sp_\ast(X) \subset f^{-1} \left( \Sp_\ast(Y) \right)  \cup \Sp_\ast(X/Y) . \]

it follows that for any $\ast \in \{ alg, ab, h \}$, $\Sp_{\ast}(X)$ is not Zariski-dense in $X$. This concludes the proof of Theorem \ref{Lang for big}.

\section{Proof of Theorem \ref{Lang for large}}

In this section, we complete the proof of Theorem \ref{Lang for large}.\\

Let $X$ be a projective complex algebraic variety equipped with a large complex local system $\cL$. Thanks to 
 Corollary \ref{Equality of special subsets}, we know that the equality $\Sp_{alg}(X) = \Sp_{ab}(X) = \Sp_h(X)$ holds. Therefore, the special subset is denoted $\Sp(X)$ without risk of confusion.\\

If $X$ is not of general type, then $\Sp(X) = X$ by Theorem \ref{Special sets of not general type varieties}. Otherwise, if $X$ is not of general type, we claim that $\Sp(X)$ is not Zariski-dense in $X$. This is an immediate consequence of Theorem \ref{Lang for big} when $X$ is a normal irreducible projective complex algebraic variety. In general, let $X = \cup_i X_i$ be the decomposition of $X$ in its irreducible components. Since $\Sp(X_i) \subset X_i$ for every $i$, it is sufficient to prove that $\Sp(X_i)$ is not Zariski dense in $X_i$ for at least one of the $X_i$'s which is of general type. Therefore one may assume that $X$ is irreducible. Let $\nu \colon \bar X \to X $ denote the normalization of $X$ and $Z \subset X$ the non-normal locus of $X$. Then $\nu(\Sp(\bar X)) \subset \Sp(X)$ since $\nu$ is finite, and $\Sp(X) \subset \nu(\Sp(\bar X))
 \cup Z$ since $\nu$ is an isomorphism onto its image outside $\nu^{-1}(Z)$. Therefore, $\Sp(X)$ is Zariski-dense in $X$ if and only if $\Sp(\bar X)$ is Zariski-dense in $\bar X$. Moreover, the pull-back of $\cL$ on $\bar X$ is a large complex local system thanks to Proposition \ref{pull-back large}. Therefore one may assume that $X$ is irreducible and normal, hence Theorem \ref{Lang for big} applies. \\

It remains to prove that $ \Sp(X)$ is Zariski-closed in $X$. There is nothing to prove when $\dim X = 0$, so that one can assume that $\dim X > 0$. By Noetherian induction, let us assume that the result holds for all strict subvarieties of $X$. If $X$ is not of general type, then $ \Sp(X) = X$ by Theorem \ref{Special sets of not general type varieties} and the result follows. Otherwise, if $X$ is of general type, then $S :=\Sp(X)$ is not Zariski-dense in $X$. Denoting by $\bar S$ its Zariski-closure in $X$, observe that $\Sp(\bar S) = \Sp(X) = S$.
Therefore, $\Sp(\bar S)$ is Zariski-dense in $\bar S$. Since $\bar S$ is a strict subvariety of $X$, it follows by Noetherian induction that $\bar S$ is not of general type. Thanks to Theorem \ref{Special sets of not general type varieties} again, the equality $\bar S = \Sp(\bar S)$ holds, so that $S = \bar S$.

\small

\bibliography{./biblio}

@article{CDY,
      title={Hyperbolicity and fundamental groups of complex quasi-projective varieties}, 
      author={Benoit Cadorel and Ya Deng and Katsutoshi Yamanoi},
      year={2024},
	journal = {arXiv e-prints},
      eprint={2212.12225},
      archivePrefix={arXiv},
      primaryClass={math.AG},
      url={https://arxiv.org/abs/2212.12225}, 
}

@book {Fujino-book,
    AUTHOR = {Fujino, Osamu},
     TITLE = {Iitaka conjecture---an introduction},
    SERIES = {SpringerBriefs in Mathematics},
 PUBLISHER = {Springer, Singapore},
      YEAR = {[2020] \copyright 2020},
     PAGES = {xiv+128},
      ISBN = {978-981-15-3347-1; 978-981-15-3346-4},
   MRCLASS = {14E30 (14D06)},
  MRNUMBER = {4177797},
MRREVIEWER = {Julien Keller},
       DOI = {10.1007/978-981-15-3347-1},
       URL = {https://doi.org/10.1007/978-981-15-3347-1},
}

@book {Iitaka-book,
    AUTHOR = {Iitaka, Shigeru},
     TITLE = {Algebraic geometry},
    SERIES = {North-Holland Mathematical Library},
    VOLUME = {24},
      NOTE = {An introduction to birational geometry of algebraic varieties},
 PUBLISHER = {Springer-Verlag, New York-Berlin},
      YEAR = {1982},
     PAGES = {x+357},
      ISBN = {0-387-90546-4},
   MRCLASS = {14-01 (14-02 14E05)},
  MRNUMBER = {637060},
MRREVIEWER = {Werner Kleinert},
}

@incollection {Kollar-subadditivity,
    AUTHOR = {Koll\'{a}r, J\'{a}nos},
     TITLE = {Subadditivity of the {K}odaira dimension: fibers of general
              type},
 BOOKTITLE = {Algebraic geometry, {S}endai, 1985},
    SERIES = {Adv. Stud. Pure Math.},
    VOLUME = {10},
     PAGES = {361--398},
 PUBLISHER = {North-Holland, Amsterdam},
      YEAR = {1987},
   MRCLASS = {14J10 (14C30 14D20)},
  MRNUMBER = {946244},
MRREVIEWER = {Yujiro Kawamata},
       DOI = {10.2969/aspm/01010361},
       URL = {https://doi.org/10.2969/aspm/01010361},
}

@article {Bloch26,
    AUTHOR = {Bloch, Andr\'e},
     TITLE = {Sur les syst\`emes de fonctions uniformes satisfaisant \`a l'\' equation d'une vari\'et\'e alg\'ebrique dont l'irr\'egularit\'e d\'epasse la dimension},
   JOURNAL = {J. Math. Pures Appl.},
      YEAR = {1926},
     PAGES = {19--66},
}

@article {Borel-Serre,
    AUTHOR = {Borel, A. and Serre, J.-P.},
     TITLE = {Cohomologie d'immeubles et de groupes {$S$}-arithm\'etiques},
   JOURNAL = {Topology},
  FJOURNAL = {Topology. An International Journal of Mathematics},
    VOLUME = {15},
      YEAR = {1976},
    NUMBER = {3},
     PAGES = {211--232},
      ISSN = {0040-9383},
   MRCLASS = {22E40},
  MRNUMBER = {447474},
MRREVIEWER = {H.\ Bass},
       DOI = {10.1016/0040-9383(76)90037-9},
       URL = {https://doi.org/10.1016/0040-9383(76)90037-9},
}

@article{Bruni-relative-Lang,
      title={The relative {G}reen-{G}riffiths-{L}ang conjecture for families of varieties of maximal {A}lbanese dimension}, 
      author={Yohan Brunebarbe},
	journal = {arXiv e-prints},
      year={2023},
      eprint={2305.09613},
      archivePrefix={arXiv},
      primaryClass={math.AG}
}

@article{Brunebarbe-Shafarevich,
      title={Existence of the {S}hafarevich morphism for semisimple local systems on quasi-projective varieties}, 
      author={Yohan Brunebarbe},
	journal = {arXiv e-prints},
      year={2023},
      eprint={2305.09741},
      archivePrefix={arXiv},
      primaryClass={math.AG}
}

@article {Campana91,
    AUTHOR = {Campana, Fr\'{e}d\'{e}ric},
     TITLE = {On twistor spaces of the class {C}},
   JOURNAL = {J. Differential Geom.},
  FJOURNAL = {Journal of Differential Geometry},
    VOLUME = {33},
      YEAR = {1991},
    NUMBER = {2},
     PAGES = {541--549},
      ISSN = {0022-040X},
   MRCLASS = {32L25 (32J20 53C25)},
  MRNUMBER = {1094468},
MRREVIEWER = {S. M. Salamon},
       URL = {http://projecteuclid.org/euclid.jdg/1214446329},
}

@article {Campana04,
    AUTHOR = {Campana, Fr\'{e}d\'{e}ric},
     TITLE = {Orbifolds, special varieties and classification theory},
   JOURNAL = {Ann. Inst. Fourier (Grenoble)},
  FJOURNAL = {Universit\'{e} de Grenoble. Annales de l'Institut Fourier},
    VOLUME = {54},
      YEAR = {2004},
    NUMBER = {3},
     PAGES = {499--630},
      ISSN = {0373-0956},
   MRCLASS = {14E05 (14D06 14J40 32Q57 35Q15)},
  MRNUMBER = {2097416},
MRREVIEWER = {Dan Abramovich},
       URL = {http://aif.cedram.org/item?id=AIF_2004__54_3_499_0},
}

@article {CCE,
    AUTHOR = {Campana, Fr\'{e}deric and Claudon, Beno\^{\i}t and Eyssidieux,
              Philippe},
     TITLE = {Repr\'{e}sentations lin\'{e}aires des groupes k\"{a}hl\'{e}riens:
              factorisations et conjecture de {S}hafarevich lin\'{e}aire},
   JOURNAL = {Compos. Math.},
  FJOURNAL = {Compositio Mathematica},
    VOLUME = {151},
      YEAR = {2015},
    NUMBER = {2},
     PAGES = {351--376},
      ISSN = {0010-437X},
   MRCLASS = {32Q15 (14D07 14E20 14F35 32Q30)},
  MRNUMBER = {3314830},
MRREVIEWER = {Andr\'{e} Oliveira},
       DOI = {10.1112/S0010437X14007751},
       URL = {https://doi.org/10.1112/S0010437X14007751},
}

@article {Corlette_JDG,
    AUTHOR = {Corlette, Kevin},
     TITLE = {Flat {$G$}-bundles with canonical metrics},
   JOURNAL = {J. Differential Geom.},
  FJOURNAL = {Journal of Differential Geometry},
    VOLUME = {28},
      YEAR = {1988},
    NUMBER = {3},
     PAGES = {361--382},
      ISSN = {0022-040X,1945-743X},
   MRCLASS = {58E20 (32L99 53C10)},
  MRNUMBER = {965220},
MRREVIEWER = {John\ C.\ Wood},
       URL = {http://projecteuclid.org/euclid.jdg/1214442469},
}

@article {Eyssidieux-reductive,
    AUTHOR = {Eyssidieux, Philippe},
     TITLE = {Sur la convexit\'{e} holomorphe des rev\^{e}tements lin\'{e}aires
              r\'{e}ductifs d'une vari\'{e}t\'{e} projective alg\'{e}brique complexe},
   JOURNAL = {Invent. Math.},
  FJOURNAL = {Inventiones Mathematicae},
    VOLUME = {156},
      YEAR = {2004},
    NUMBER = {3},
     PAGES = {503--564},
      ISSN = {0020-9910},
   MRCLASS = {32E05 (14E20 32J25 32Q30)},
  MRNUMBER = {2061328},
MRREVIEWER = {Ludmil Katzarkov},
       DOI = {10.1007/s00222-003-0345-0},
       URL = {https://doi.org/10.1007/s00222-003-0345-0},
}

@article {EKPR,
    AUTHOR = {Eyssidieux, P. and Katzarkov, L. and Pantev, T. and
              Ramachandran, M.},
     TITLE = {Linear {S}hafarevich conjecture},
   JOURNAL = {Ann. of Math. (2)},
  FJOURNAL = {Annals of Mathematics. Second Series},
    VOLUME = {176},
      YEAR = {2012},
    NUMBER = {3},
     PAGES = {1545--1581},
      ISSN = {0003-486X},
   MRCLASS = {14C30 (14D07 14D20 14E20 14F35 32E05 32Q30 58A14)},
  MRNUMBER = {2979857},
MRREVIEWER = {Azniv Kasparian},
       DOI = {10.4007/annals.2012.176.3.4},
       URL = {https://doi.org/10.4007/annals.2012.176.3.4},
}

@incollection {Green-Griffiths79,
    AUTHOR = {Green, Mark and Griffiths, Phillip},
     TITLE = {Two applications of algebraic geometry to entire holomorphic
              mappings},
 BOOKTITLE = {The {C}hern {S}ymposium 1979 ({P}roc. {I}nternat. {S}ympos.,
              {B}erkeley, {C}alif., 1979)},
     PAGES = {41--74},
 PUBLISHER = {Springer, New York-Berlin},
      YEAR = {1980},
   MRCLASS = {32H25 (14J25 32J15)},
  MRNUMBER = {609557},
MRREVIEWER = {R. R. Simha},
}

@article {Griffiths-Schmid,
    AUTHOR = {Griffiths, Phillip and Schmid, Wilfried},
     TITLE = {Locally homogeneous complex manifolds},
   JOURNAL = {Acta Math.},
  FJOURNAL = {Acta Mathematica},
    VOLUME = {123},
      YEAR = {1969},
     PAGES = {253--302},
      ISSN = {0001-5962},
   MRCLASS = {57.60 (32.00)},
  MRNUMBER = {259958},
MRREVIEWER = {Shoshichi Kobayashi},
       DOI = {10.1007/BF02392390},
       URL = {https://doi.org/10.1007/BF02392390},
}

@article {Gromov-Schoen,
    AUTHOR = {Gromov, Mikhail and Schoen, Richard},
     TITLE = {Harmonic maps into singular spaces and {$p$}-adic
              superrigidity for lattices in groups of rank one},
   JOURNAL = {Inst. Hautes \'{E}tudes Sci. Publ. Math.},
  FJOURNAL = {Institut des Hautes \'{E}tudes Scientifiques. Publications
              Math\'{e}matiques},
    NUMBER = {76},
      YEAR = {1992},
     PAGES = {165--246},
      ISSN = {0073-8301},
   MRCLASS = {58E20 (22E40)},
  MRNUMBER = {1215595},
MRREVIEWER = {Caio J. C. Negreiros},
       URL = {http://www.numdam.org/item?id=PMIHES_1992__76__165_0},
}

@article {Javanpeykar-Rousseau,
    AUTHOR = {Javanpeykar, Ariyan and Rousseau, Erwan},
     TITLE = {Albanese maps and fundamental groups of varieties with many
              rational points over function fields},
   JOURNAL = {Int. Math. Res. Not. IMRN},
  FJOURNAL = {International Mathematics Research Notices. IMRN},
      YEAR = {2022},
    NUMBER = {24},
     PAGES = {19354--19398},
      ISSN = {1073-7928,1687-0247},
   MRCLASS = {14F35 (11G35 14G05)},
  MRNUMBER = {4523251},
MRREVIEWER = {Chunhui\ Liu},
       DOI = {10.1093/imrn/rnab255},
       URL = {https://doi.org/10.1093/imrn/rnab255},
}

@incollection {Katzarkov,
    AUTHOR = {Katzarkov, L.},
     TITLE = {On the {S}hafarevich maps},
 BOOKTITLE = {Algebraic geometry---{S}anta {C}ruz 1995},
    SERIES = {Proc. Sympos. Pure Math.},
    VOLUME = {62},
     PAGES = {173--216},
 PUBLISHER = {Amer. Math. Soc., Providence, RI},
      YEAR = {1997},
   MRCLASS = {14C30 (14E20 14F45 32E10)},
  MRNUMBER = {1492537},
MRREVIEWER = {Allen B. Altman},
       DOI = {10.1090/pspum/062.2/1492537},
       URL = {https://doi.org/10.1090/pspum/062.2/1492537},
}

@article {Kawamata80,
    AUTHOR = {Kawamata, Yujiro},
     TITLE = {On {B}loch's conjecture},
   JOURNAL = {Invent. Math.},
  FJOURNAL = {Inventiones Mathematicae},
    VOLUME = {57},
      YEAR = {1980},
    NUMBER = {1},
     PAGES = {97--100},
      ISSN = {0020-9910},
     CODEN = {INVMBH},
   MRCLASS = {32H99},
  MRNUMBER = {564186},
MRREVIEWER = {Toshiyuki Maebashi},
       DOI = {10.1007/BF01389820},
       URL = {http://dx.doi.org/10.1007/BF01389820},
}

@article {Kawamata-characterization,
    AUTHOR = {Kawamata, Yujiro},
     TITLE = {Characterization of abelian varieties},
   JOURNAL = {Compositio Math.},
  FJOURNAL = {Compositio Mathematica},
    VOLUME = {43},
      YEAR = {1981},
    NUMBER = {2},
     PAGES = {253--276},
      ISSN = {0010-437X},
   MRCLASS = {14J10 (32J15)},
  MRNUMBER = {622451},
MRREVIEWER = {Daniel Comenetz},
       URL = {http://www.numdam.org/item?id=CM_1981__43_2_253_0},
}

@article {Klingler2003,
    AUTHOR = {Klingler, Bruno},
     TITLE = {Sur la rigidit\'{e} de certains groupes fondamentaux,
              l'arithm\'{e}ticit\'{e} des r\'{e}seaux hyperboliques complexes, et les
              ``faux plans projectifs''},
   JOURNAL = {Invent. Math.},
  FJOURNAL = {Inventiones Mathematicae},
    VOLUME = {153},
      YEAR = {2003},
    NUMBER = {1},
     PAGES = {105--143},
      ISSN = {0020-9910},
   MRCLASS = {22E40 (14J29 32J15 32Q05 32Q45 53C24 58E20)},
  MRNUMBER = {1990668},
MRREVIEWER = {Philippe P. Eyssidieux},
       DOI = {10.1007/s00222-002-0283-2},
       URL = {https://doi.org/10.1007/s00222-002-0283-2},
}

@article {Kollar-Shafarevich,
    AUTHOR = {Koll\'ar, J\'anos},
     TITLE = {Shafarevich maps and plurigenera of algebraic varieties},
   JOURNAL = {Invent. Math.},
  FJOURNAL = {Inventiones Mathematicae},
    VOLUME = {113},
      YEAR = {1993},
    NUMBER = {1},
     PAGES = {177--215},
      ISSN = {0020-9910,1432-1297},
   MRCLASS = {14E20 (14E30 14J10)},
  MRNUMBER = {1223229},
MRREVIEWER = {Alessio\ Corti},
       DOI = {10.1007/BF01244307},
       URL = {https://doi.org/10.1007/BF01244307},
}

@book {Kollar-Shafa-book,
    AUTHOR = {Koll\'ar, J\'anos},
     TITLE = {Shafarevich maps and automorphic forms},
    SERIES = {M. B. Porter Lectures},
 PUBLISHER = {Princeton University Press, Princeton, NJ},
      YEAR = {1995},
     PAGES = {x+201},
      ISBN = {0-691-04381-7},
   MRCLASS = {14E20 (14J10 14J15 32J18 32N10)},
  MRNUMBER = {1341589},
MRREVIEWER = {Kang\ Zuo},
       DOI = {10.1515/9781400864195},
       URL = {https://doi.org/10.1515/9781400864195},
}

@article {Lubotzky-Magid,
    AUTHOR = {Lubotzky, Alexander and Magid, Andy R.},
     TITLE = {Varieties of representations of finitely generated groups},
   JOURNAL = {Mem. Amer. Math. Soc.},
  FJOURNAL = {Memoirs of the American Mathematical Society},
    VOLUME = {58},
      YEAR = {1985},
    NUMBER = {336},
     PAGES = {xi+117},
      ISSN = {0065-9266},
   MRCLASS = {20C15 (14D20)},
  MRNUMBER = {818915},
MRREVIEWER = {Avner Ash},
       DOI = {10.1090/memo/0336},
       URL = {https://doi.org/10.1090/memo/0336},
}

@book {Nakayama,
    AUTHOR = {Nakayama, Noboru},
     TITLE = {Zariski-decomposition and abundance},
    SERIES = {MSJ Memoirs},
    VOLUME = {14},
 PUBLISHER = {Mathematical Society of Japan, Tokyo},
      YEAR = {2004},
     PAGES = {xiv+277},
      ISBN = {4-931469-31-0},
   MRCLASS = {14C20 (14E15 14E30 14J10)},
  MRNUMBER = {2104208},
MRREVIEWER = {Tommaso De Fernex},
}

@article {Lang86,
    AUTHOR = {Lang, Serge},
     TITLE = {Hyperbolic and {D}iophantine analysis},
   JOURNAL = {Bull. Amer. Math. Soc. (N.S.)},
  FJOURNAL = {American Mathematical Society. Bulletin. New Series},
    VOLUME = {14},
      YEAR = {1986},
    NUMBER = {2},
     PAGES = {159--205},
      ISSN = {0273-0979},
     CODEN = {BAMOAD},
   MRCLASS = {32H20 (11D99 14G05)},
  MRNUMBER = {828820},
MRREVIEWER = {Junjiro Noguchi},
       DOI = {10.1090/S0273-0979-1986-15426-1},
       URL = {http://dx.doi.org/10.1090/S0273-0979-1986-15426-1},
}

@article {Ochiai77,
    AUTHOR = {Ochiai, Takushiro},
     TITLE = {On holomorphic curves in algebraic varieties with ample
              irregularity},
   JOURNAL = {Invent. Math.},
  FJOURNAL = {Inventiones Mathematicae},
    VOLUME = {43},
      YEAR = {1977},
    NUMBER = {1},
     PAGES = {83--96},
      ISSN = {0020-9910},
   MRCLASS = {32H25},
  MRNUMBER = {0473237},
MRREVIEWER = {Marcus Wright},
}

@article {Sikora,
    AUTHOR = {Sikora, Adam S.},
     TITLE = {Character varieties},
   JOURNAL = {Trans. Amer. Math. Soc.},
  FJOURNAL = {Transactions of the American Mathematical Society},
    VOLUME = {364},
      YEAR = {2012},
    NUMBER = {10},
     PAGES = {5173--5208},
      ISSN = {0002-9947},
   MRCLASS = {14D20 (14L24 53D30 57M50)},
  MRNUMBER = {2931326},
MRREVIEWER = {Benjamin M. S. Martin},
       DOI = {10.1090/S0002-9947-2012-05448-1},
       URL = {https://doi.org/10.1090/S0002-9947-2012-05448-1},
}

@article {SimpsonHiggs,
    AUTHOR = {Simpson, Carlos T.},
     TITLE = {Higgs bundles and local systems},
   JOURNAL = {Inst. Hautes \'Etudes Sci. Publ. Math.},
  FJOURNAL = {Institut des Hautes \'Etudes Scientifiques. Publications
              Math\'ematiques},
    NUMBER = {75},
      YEAR = {1992},
     PAGES = {5--95},
      ISSN = {0073-8301},
     CODEN = {PMIHA6},
   MRCLASS = {32G13 (14D07 53C07 58D27 58E15)},
  MRNUMBER = {1179076 (94d:32027)},
MRREVIEWER = {William Goldman},
       URL = {http://www.numdam.org/item?id=PMIHES_1992__75__5_0},
}

@book {Ueno,
    AUTHOR = {Ueno, Kenji},
     TITLE = {Classification theory of algebraic varieties and compact
              complex spaces},
    SERIES = {Lecture Notes in Mathematics, Vol. 439},
      NOTE = {Notes written in collaboration with P. Cherenack},
 PUBLISHER = {Springer-Verlag, Berlin-New York},
      YEAR = {1975},
     PAGES = {xix+278},
   MRCLASS = {14D20 (32J25)},
  MRNUMBER = {0506253},
MRREVIEWER = {D. Lieberman},
}

@incollection {Viehweg83,
    AUTHOR = {Viehweg, Eckart},
     TITLE = {Weak positivity and the additivity of the {K}odaira dimension
              for certain fibre spaces},
 BOOKTITLE = {Algebraic varieties and analytic varieties ({T}okyo, 1981)},
    SERIES = {Adv. Stud. Pure Math.},
    VOLUME = {1},
     PAGES = {329--353},
 PUBLISHER = {North-Holland, Amsterdam},
      YEAR = {1983},
   MRCLASS = {14J10 (14D20 14F05)},
  MRNUMBER = {715656},
MRREVIEWER = {Yujiro Kawamata},
}

@article {Yamanoi-Fourier,
    AUTHOR = {Yamanoi, Katsutoshi},
     TITLE = {On fundamental groups of algebraic varieties and value
              distribution theory},
   JOURNAL = {Ann. Inst. Fourier (Grenoble)},
  FJOURNAL = {Universit\'{e} de Grenoble. Annales de l'Institut Fourier},
    VOLUME = {60},
      YEAR = {2010},
    NUMBER = {2},
     PAGES = {551--563},
      ISSN = {0373-0956},
   MRCLASS = {32H30 (14F35)},
  MRNUMBER = {2667786},
MRREVIEWER = {William A. Cherry},
       URL = {http://aif.cedram.org/item?id=AIF_2010__60_2_551_0},
}

@incollection {Yamanoi-lectures,
    AUTHOR = {Yamanoi, Katsutoshi},
     TITLE = {Kobayashi hyperbolicity and higher-dimensional {N}evanlinna
              theory},
 BOOKTITLE = {Geometry and analysis on manifolds},
    SERIES = {Progr. Math.},
    VOLUME = {308},
     PAGES = {209--273},
 PUBLISHER = {Birkh\"{a}user/Springer, Cham},
      YEAR = {2015},
   MRCLASS = {32Q45 (32H30)},
  MRNUMBER = {3331401},
MRREVIEWER = {Pei-Chu Hu},
       DOI = {10.1007/978-3-319-11523-8\_9},
       URL = {https://doi.org/10.1007/978-3-319-11523-8_9},
}

@article {Yamanoi-maximal-Albanese,
    AUTHOR = {Yamanoi, Katsutoshi},
     TITLE = {Holomorphic curves in algebraic varieties of maximal
              {A}lbanese dimension},
   JOURNAL = {Internat. J. Math.},
  FJOURNAL = {International Journal of Mathematics},
    VOLUME = {26},
      YEAR = {2015},
    NUMBER = {6},
     PAGES = {1541006, 45},
      ISSN = {0129-167X},
   MRCLASS = {32H30 (14K20)},
  MRNUMBER = {3356877},
MRREVIEWER = {Qiming Yan},
       DOI = {10.1142/S0129167X15410062},
       URL = {https://doi.org/10.1142/S0129167X15410062},
}

@article {Zuo-Chern-hyperbolicity,
    AUTHOR = {Zuo, Kang},
     TITLE = {Kodaira dimension and {C}hern hyperbolicity of the
              {S}hafarevich maps for representations of {$\pi_1$} of compact
              {K}\"{a}hler manifolds},
   JOURNAL = {J. Reine Angew. Math.},
  FJOURNAL = {Journal f\"{u}r die Reine und Angewandte Mathematik. [Crelle's
              Journal]},
    VOLUME = {472},
      YEAR = {1996},
     PAGES = {139--156},
      ISSN = {0075-4102},
   MRCLASS = {32J27 (14D07 32J25 58E20)},
  MRNUMBER = {1384908},
MRREVIEWER = {Kimio Miyajima},
       DOI = {10.1515/crll.1996.472.139},
       URL = {https://doi.org/10.1515/crll.1996.472.139},
}

@book {Zuo-book,
    AUTHOR = {Zuo, Kang},
     TITLE = {Representations of fundamental groups of algebraic varieties},
    SERIES = {Lecture Notes in Mathematics},
    VOLUME = {1708},
 PUBLISHER = {Springer-Verlag, Berlin},
      YEAR = {1999},
     PAGES = {viii+135},
      ISBN = {3-540-66312-6},
   MRCLASS = {14F35 (14E20 32J27 58E20)},
  MRNUMBER = {1738433},
       DOI = {10.1007/BFb0092569},
       URL = {https://doi.org/10.1007/BFb0092569},
}

\bibliographystyle{amsalpha}

\end{document}